\documentclass[11pt]{amsart}
\usepackage[utf8]{inputenc}
\usepackage{amscd,amsmath,amssymb,amsfonts,verbatim}
\usepackage[cmtip, all]{xy}
\usepackage{comment}

\usepackage{hyperref}

\numberwithin{equation}{section}

\newtheorem{thm}[equation]{Theorem}
\newtheorem{prop}[equation]{Proposition}
\newtheorem{lem}[equation]{Lemma}
\newtheorem{cor}[equation]{Corollary}

\theoremstyle{definition}
\newtheorem{defn}[equation]{Definition}
\theoremstyle{remark}
\newtheorem{remk}[equation]{Remark}
\newtheorem{remks}[equation]{Remarks}

\newtheorem{exm}[equation]{Example}
\newtheorem{exms}[equation]{Examples}
\newtheorem{notat}[equation]{Notation}

{\hfill$\square$\end{defn}}
\newenvironment{rem}{\begin{remk}}%
{\hfill$\square$\end{remk}}
{\hfill$\square$\end{remks}}
\newenvironment{ex}{\begin{exm}}%
{\hfill$\square$\end{exm}}
{\hfill$\square$\end{exms}}
{\hfill$\square$\end{notat}}

\newcommand{\thmref}{Theorem~\ref}
\newcommand{\propref}{Proposition~\ref}
\newcommand{\corref}{Corollary~\ref}
\newcommand{\defref}{Definition~\ref}
\newcommand{\lemref}{Lemma~\ref}

\newcommand{\sA}{{\mathcal A}}
\newcommand{\sB}{{\mathcal B}}

\newcommand{\sE}{{\mathcal E}}
\newcommand{\sF}{{\mathcal F}}
\newcommand{\sG}{{\mathcal G}}
\newcommand{\sH}{{\mathcal H}}
\newcommand{\sI}{{\mathcal I}}

\newcommand{\sL}{{\mathcal L}}

\newcommand{\sO}{{\mathcal O}}
\newcommand{\sP}{{\mathcal P}}
\newcommand{\sQ}{{\mathcal Q}}

\newcommand{\sS}{{\mathcal S}}

\newcommand{\sU}{{\mathcal U}}
\newcommand{\sV}{{\mathcal V}}
\newcommand{\sW}{{\mathcal W}}
\newcommand{\sX}{{\mathcal X}}
\newcommand{\sY}{{\mathcal Y}}
\newcommand{\sZ}{{\mathcal Z}}
\newcommand{\A}{{\mathbb A}}

\newcommand{\G}{{\mathbb G}}

\newcommand{\N}{{\mathbb N}}
\renewcommand{\P}{{\mathbb P}}

\newcommand{\Z}{{\mathbb Z}}

\newcommand{\surj}{\twoheadrightarrow}
\newcommand{\inj}{\hookrightarrow}

\newcommand{\Hom}{{\rm Hom}}

\newcommand{\Spec}{{\rm Spec}}

\newcommand{\qc}{{\rm qc}}
\newcommand{\pc}{{\rm pc}}
\newcommand{\perf}{{\rm perf}}

\newcommand{\holim}{\mathop{{\rm holim}}}

\newcommand{\GL}{{\operatorname{\rm GL}}}

\newcommand{\Nis}{{\operatorname{Nis}}}

\newcommand{\ds}{{/\kern-3pt/}}

\newcommand{\Proj}{{\operatorname{Proj}}}

\newcommand{\colim}{\mathop{\text{\rm colim}}}

\newcommand{\Mod}{{\mathrm{Mod}}}
\newcommand{\Coh}{{\mathrm{Coh}}}
\newcommand{\Ch}{{\mathrm{Ch}}}
\newcommand{\QC}{{\mathrm{QCoh}}}
\newcommand{\LisEt}{{\mathrm{Lis}\textrm{-}\mathrm{\acute Et}}}

\renewcommand{\dim}{\text{\rm dim}}

\newcommand{\tuborg}{\left\{\begin{array}{ll}}
\newcommand{\sluttuborg}{\end{array}\right.}

\newcommand{\wt}{\widetilde}

\DeclareMathOperator*{\hocolim}{hocolim}

\begin{document}
 
\title{Vanishing theorems for the negative $K$-theory of stacks}
\author{Marc Hoyois and Amalendu Krishna}
\address{Department of Mathematics, University of Southern California, 
Los Angeles, CA, USA}
\email{hoyois@usc.edu}
\address{School of Mathematics, Tata Institute of Fundamental Research,  
1 Homi Bhabha Road, Colaba, Mumbai, India}
\email{amal@math.tifr.res.in}
\thanks{The first author was partially supported by NSF grant DMS-1508096}

\keywords{Algebraic $K$-theory, Negative $K$-theory, 
Algebraic stacks}

\subjclass[2010]{Primary 19D35; Secondary 14D23}

\begin{abstract}
We prove that the homotopy algebraic $K$-theory of 
tame quasi-DM stacks satisfies cdh-descent. We apply this descent result to prove 
that if $\sX$ is a Noetherian tame quasi-DM stack and $i<-\dim(\sX)$, then 
$K_i(\sX)[1/n]=0$ (resp.\ $K_i(\sX,\Z/n)=0$) provided that
$n$ is nilpotent on $\sX$ (resp.\ is invertible on $\sX$).
Our descent and vanishing results apply more generally to certain Artin stacks
whose stabilizers are extensions of finite group schemes by group schemes
of multiplicative type.
\end{abstract} 

\setcounter{tocdepth}{2}
\maketitle
\tableofcontents  

\section{Introduction}\label{sec:Intro}
The negative $K$-theory of rings was defined by Bass \cite{Bass} and it was 
later generalized to all schemes by Thomason and Trobaugh \cite{TT},
who established its fundamental properties such as
localization, excision, Mayer--Vietoris, and the projective bundle formula.

As explained in \cite{TT}, these properties of $K$-theory give rise
to the Bass--Thomason--Trobaugh non-connective $K$-theory, or $K^B$-theory,
 which is usually non-trivial in negative degrees for singular schemes.
  A famous conjecture of Weibel
asserts that for a Noetherian scheme $X$ of Krull dimension $d$, the group 
$K_i(X)$ vanishes for $i < -d$. This conjecture was settled by
Weibel for excellent surfaces \cite{Weibel-3}, by
Corti{\~n}as, Haesemeyer, Schlichting and Weibel for schemes essentially of finite type
over a field of characteristic zero \cite{CHSW}, and recently by Kerz, Strunk and Tamme
for all Noetherian schemes \cite{KST}.

Before a complete proof of Weibel's conjecture for schemes appeared in
\cite{KST}, Kelly \cite{Kelly} used the alteration methods of de Jong and 
Gabber to show that the vanishing conjecture for negative 
$K$-theory holds in characteristic $p > 0$ if one is allowed to invert $p$. 
Later, Kerz and Strunk \cite{KS} gave a different proof of Kelly's theorem
by proving Weibel's conjecture for negative homotopy $K$-theory, or $KH$-theory, a variant
of $K$-theory introduced by Weibel \cite{Weibel-1}.
In their proof, Kerz and Strunk used the method of flatification by blow-up
instead of alterations. 

It is natural to ask for an extension of Weibel's conjecture to algebraic stacks. 
The algebraic $K$-theory of quotient stacks was
introduced by Thomason \cite{Thom} in order to study algebraic $K$-theory
of a scheme which can be equipped with an action of a group
scheme. The localization, excision, and 
Mayer--Vietoris properties for the algebraic $K$-theory of tame Deligne--Mumford stacks
were proven by the second author and {\O}stv{\ae}r \cite{KO}, and
 together with the projective bundle formula they were established for 
more general quotient stacks by the second author and Ravi \cite{KR-2}. 
The $K^B$-theory of Bass--Thomason--Trobaugh and the $KH$-theory of Weibel
 were also generalized to such quotient stacks in \cite{KR-2}.

The purpose of this paper is to show that the approach of Kerz and Strunk 
can be generalized to a large class
of algebraic stacks, including all tame Artin stacks in the sense of \cite{AOV}.
As a consequence, we will obtain a generalization of Kelly's vanishing theorem for
the negative $K$-theory of such stacks.

\subsection{Vanishing of negative $K$-theory of 
stacks}\label{sec:Neg-K*}

Our main results apply to certain algebraic stacks with finite or 
multiplicative type stabilizers. More precisely, 
let $\mathbf{Stk}'$ be the category consisting of the following algebraic stacks:
\begin{itemize}
	\item stacks with separated diagonal and linearly reductive finite stabilizers;
	\item stacks with affine diagonal whose stabilizers are
	extensions of linearly reductive finite groups by groups of multiplicative type.
\end{itemize}
Note that $\mathbf{Stk}'$ contains tame Artin stacks with separated diagonal in the sense of \cite{AOV}.
The {\sl blow-up dimension} of a Noetherian stack $\sX$ is a modification of the Krull dimension which is invariant
under blow-ups (see \defref{defn:cov-dim}); it coincides with the usual dimension when $\sX$ is a quasi-DM stack.

\begin{thm}[See Theorems \ref{thm:Vanishing-KH}, \ref{thm:DM-stack}, and \ref{thm:Vanishing-K}]
	\label{thm:Main-1}
	Let $\sX$ be a stack in $\mathbf{Stk}'$ satisfying the resolution property
	or having finite inertia. 
	Assume that $\sX$ is Noetherian of blow-up dimension $d$.
Then the following hold.
\begin{enumerate}
\item
$KH_i(\sX) = 0$ for $i < -d$.
\item If $n$ is nilpotent on $\sX$,
$K_i(\sX)[{1}/{n}] = 0$ for $i < -d$.
\item If $n$ is invertible on $\sX$,
$K_i(\sX, {\Z}/n) = 0$ for $i < -d$.
\end{enumerate}
\end{thm}

\subsection{Cdh-descent for the homotopy $K$-theory of stacks}\label{sec:cdh-DM*}
Cdh-descent plays a key role in all the existing
vanishing theorems for negative $K$-theory. In the recent proof of Weibel's conjecture
in \cite{KST}, the central result is pro-cdh-descent for non-connective
algebraic $K$-theory.
Earlier results towards Weibel's conjecture used instead
cdh-descent for homotopy $K$-theory $KH$.
For schemes over a field of characteristic zero, this descent result
was proven by Haesemeyer \cite{Haes}, and in arbitrary characteristic,
it was shown by Cisinski \cite{Cisinski}. For the equivariant $KH$-theory
of quasi-projective schemes acted on by a diagonalizable or finite linearly reductive group
over an arbitrary base, cdh-descent was proven by the first author \cite{Hoyois-2}. 
A key step in the proof of~\thmref{thm:Main-1} is a generalization of
the latter to more general algebraic stacks:

\begin{thm}[See~\thmref{thm:cdh-descent}]
	\label{thm:Main-2}
The presheaf of homotopy $K$-theory spectra $KH$
satisfies cdh-descent on the category $\mathbf{Stk}'$.
\end{thm}

Cdh-descent is the combination of two descent properties:
descent for the Nisnevich topology and descent for abstract blow-ups. 
Descent for the Nisnevich topology holds
much more generally (see \corref{cor:Nis-descent-KH}) and in fact it holds
for non-connective $K$-theory as well (see \corref{cor:Nis-descent-KB}).
Descent for abstract blow-ups is more difficult and uses several non-trivial properties
of the category $\mathbf{Stk}'$. The proof ultimately relies
on the proper base change theorem in stable equivariant motivic homotopy theory,
proved in \cite{Hoyois-1}.

\section{Preliminaries on algebraic stacks}
\label{sec:stacks}

A {\sl stack} in this text will mean a quasi-compact and quasi-separated
algebraic stack. 
Note that all morphisms between such stacks are quasi-compact and quasi-separated.
Similarly, algebraic spaces and schemes are always assumed
to be quasi-compact and quasi-separated.
We will say that a morphism of stacks is {\sl representable} (resp.\ {\sl schematic})
if it is representable by algebraic spaces (resp.\ by schemes).
Recall that the diagonal of a stack is
representable by definition (see \cite[Tag 026N]{SP}).
If $\sX$ is a stack, $k$ is a field, and $x:\Spec(k)\to\sX$ is a $k$-point, the
stabilizer $G_x\to\Spec(k)$ is a flat separated group scheme of finite type (see \cite[Tag 0B8D]{SP}).

All group schemes will be assumed flat and finitely presented.
With this convention, if $G$ is a group scheme over a scheme $S$, then $\sB G=[S/G]$ is a stack.
Recall that $G$ is called {\sl linearly reductive} if the pushforward functor
$\QC({\sB G}) \to \QC(S)$ on quasi-coherent sheaves is exact.
One knows from \cite[Theorem~2.16]{AOV} that a finite étale group scheme $G$ over $S$ is linearly reductive
if and only if its degree at each point of $S$ is prime to the residual characteristic.
Diagonalizable group schemes are also linearly reductive by \cite[Exp. I, Th. 5.3.3]{SGA3}. As linear reductivity
 is an fpqc-local property on $S$ \cite[Proposition 2.4]{AOV}, every group scheme of multiplicative type is linearly reductive.
 
 We will say that a group scheme $G$ is {\sl almost multiplicative} if it is an
 extension of a finite étale group scheme by a group scheme of multiplicative type.
 Since the class of linearly reductive group schemes is closed under quotients and extensions \cite[Proposition 12.17]{Alper},
 an almost multiplicative group scheme $G$ over $S$ is linearly reductive if and only if,
 for every $s\in S$, the number of geometric components of $G_s$ is invertible in $\kappa(s)$.

\subsection{Quasi-projective morphisms}

Recall from \cite[\S14.3]{LM} that if $\sX$ is a stack and if $\sA^{\bullet}$ is
a quasi-coherent sheaf of graded $\sO_{\sX}$-algebras, then 
$\Proj(\sA^{\bullet})$ is a local construction on the fppf site of
$\sX$ just like for schemes and hence defines a schematic
morphism of stacks $q:\Proj(\sA^{\bullet}) \to \sX$.
A morphism $f: \sY \to \sX$ is called {\sl quasi-projective} (see \cite[\S14.3.4]{LM} and \cite[Theorem 8.6]{Rydh-approx})
if there is a finitely generated quasi-coherent
sheaf $\sE$ on $\sX$ and a factorization 
$$\sY \stackrel{\iota}{\inj} \P(\sE) \xrightarrow{q} \sX$$
of $f$, where $\P(\sE)=\Proj({\rm Sym}^{\bullet}(\sE))$ 
and $\iota$ is a quasi-compact immersion.
We say that $f$ is {\sl projective} if it is quasi-projective and proper. It is clear that a quasi-projective
morphism of stacks is schematic and hence representable.

\begin{lem}\label{lem:qproj-composition}
	If $f:\sX\to\sY$ and $g:\sY\to\sZ$ are quasi-projective (resp.\ projective) morphisms of stacks,
	then $g\circ f$ is quasi-projective (resp.\ projective).
\end{lem}

\begin{proof}
The proof is the same as \cite[Lemma 2.13]{Hoyois-1}, 
the key point being that every quasi-coherent
sheaf on a quasi-compact quasi-separated stack is the 
colimit of its finitely generated quasi-coherent subsheaves
\cite{Rydh-approx}.
\end{proof}

If $\sI \subset \sO_{\sX}$ is a finitely generated quasi-coherent sheaf of ideals, defining a
finitely presented closed substack $\sZ \subset \sX$, then 
$\Proj(\oplus_{i \ge 0} \sI^i) = {\rm Bl}_{\sZ}(\sX)$ is called the
blow-up of $\sX$ with center $\sZ$. 
Note that ${\rm Bl}_\sZ(\sX)$ is a closed substack of $\P(\sI)$.
Since $\sI$ is finitely generated, it follows that the structure map ${\rm Bl}_{\sZ}(\sX)\to\sX$ is projective.
 If $\sU\subset\sX$ is an open substack, we say that a blow-up of $\sX$
 is {\sl $\sU$-admissible} if its center is disjoint from $\sU$.

\subsection{Flatification by blow-ups}

\begin{thm}[Rydh]\label{thm:Rydh}
Let $\sS$ be a quasi-compact and quasi-separated algebraic stack and let 
$f: \sX \to \sS$ be a morphism of finite type.
Let $\sF$ be a finitely generated quasi-coherent $\sO_{\sX}$-module. Let
$\sU \subseteq \sS$ be an open substack such that $f|_{\sU}$ is of finite
presentation and $\sF |_{f^{-1}(\sU)}$ is
of finite presentation and flat over $\sU$. Then there exists a sequence of
$\sU$-admissible blow-ups $\wt{\sS} \to \sS$ such that the
strict transform of $\sF$ is of finite presentation and flat over $\wt{\sS}$.
\end{thm}

\begin{proof}
	This is proved in \cite[Theorem 4.2]{Rydh-2}.
\end{proof}

\begin{lem}\label{lem:flat-birational}
	Let $f\colon\sY\to\sX$ be a flat, proper, finitely presented, representable and birational morphism of stacks.
	Then $f$ is an isomorphism.
\end{lem}

\begin{proof}
	We can assume that $\sX$ and hence $\sY$ are algebraic spaces. 
	Since $f$ is flat, proper, and finitely presented, its fibers have locally constant dimension \cite[Tag 0D4R]{SP}.
	Since $f$ is birational, its fibers must have dimension $0$, so $f$ is quasi-finite \cite[Tag 04NV]{SP}.
	By Zariski's Main Theorem \cite[Tag 082K]{SP}, we deduce that $f$ is in fact finite.
	Being finite, flat, and finitely presented, $f$ is locally free, and it must be of rank $0$.
\end{proof}

\begin{cor}[Rydh]\label{cor:birational-Chow}
Let $f:\sY\to\sX$ be a proper representable morphism of stacks
that is an isomorphism over some quasi-compact open substack $\sU\subset\sX$.
 Then there exists a projective morphism $g:\tilde\sY\to\sY$ that is an isomorphism over $\sU$
 such that $f\circ g$ is also projective.
\end{cor}

\begin{proof}
By first blowing up a finitely presented complement of $\sU$ in $\sX$ 
(which exists by \cite[Proposition 8.2]{Rydh-approx}) and replacing $\sY$ by its strict transform,
we may assume that $\sU$ is dense in $\sX$.
By~\thmref{thm:Rydh}, we can find a sequence of $\sU$-admissible blow-ups $\tilde\sX\to\sX$ 
such that the strict transform $\tilde f:\tilde\sY\to\tilde\sX$ is flat and of finite presentation.
Let $g:\tilde\sY\to\sY$ be the induced map:
 \begin{equation*}
 \xymatrix{
 \sU \ar@{^{(}->}[r] \ar@{=}[d] & \tilde\sY \ar[d]^{\tilde f} \ar[r]^g & \sY \ar[d]^{f} \\  
 \sU \ar@{^{(}->}[r]  & \tilde\sX \ar[r] & \sX.}
 \end{equation*}
 Then $g$ is a sequence of $\sU$-admissible blow-ups and hence
 it is projective by \lemref{lem:qproj-composition}.
 Moreover, $\tilde f$ is flat, proper, finitely presented, representable, and birational, whence
 an isomorphism (\lemref{lem:flat-birational}).
 Thus, $f\circ g$ is the composition of an isomorphism and the sequence of blow-ups
 $\tilde\sX\to\sX$, so it is projective by~\lemref{lem:qproj-composition}.
\end{proof}

\subsection{Nisnevich coverings of stacks}\label{sec:Nisnevich}

The following definition appears in \cite[Definition 3.1]{HR-3} and,
 for Deligne--Mumford stacks, in \cite[Definition 6.3]{KO}.

\begin{defn}\label{defn:Nis-top}
	Let $\sX$ be a stack. A family of étale morphisms $\{\sU_i\to\sX\}_{i\in I}$
	is called a {\sl Nisnevich covering} if, for every $x\in\sX$, there exists $i\in I$
	and $u\in\sU_i$ above $x$ such that the induced morphism of residual gerbes
	$\eta_u\to\eta_x$ is an isomorphism.
\end{defn}

Let $f: \sY\to\sX$ be a morphism of stacks. A {\sl monomorphic splitting sequence}
for $f$ is a sequence of quasi-compact open substacks
\[
\emptyset=\sU_0\subset\sU_1\subset\dotsb\subset\sU_n=\sX
\]
such that $f$ admits a monomorphic section over the reduced substack $\sU_i\setminus\sU_{i-1}$ for all $i$.
Note that if $f$ is étale, such a section is an open immersion $\sU_i\setminus\sU_{i-1}\inj\sY\times_\sX(\sU_i\setminus\sU_{i-1})$.

\begin{prop}\label{prop:splitting-sequence}
	Let $\sX$ be a stack. A family of étale morphisms $\{\sU_i\to\sX\}_i$ is a Nisnevich covering
	if and only if the morphism $\coprod_i\sU_i\to\sX$ admits a monomorphic splitting sequence.
\end{prop}

\begin{proof}
	See \cite[Proposition 3.3]{HR-3}.
\end{proof}

\begin{cor}\label{cor:Nis-qc}
	Let $\sX$ be a stack and let $\{\sU_i\to\sX\}_{i\in I}$ be a Nisnevich covering.
	Then there exists a finite subset $J\subset I$ such that
	$\{\sU_i\to\sX\}_{i\in J}$ is a Nisnevich covering.
\end{cor}

\begin{proof}
	This follows at once from~\propref{prop:splitting-sequence}.
\end{proof}

A {\sl Nisnevich square} in the
category of stacks is a Cartesian square of the form
\begin{equation}\label{eqn:Nis-square}
\xymatrix{
\sW \ar@{^{(}->}[r] \ar[d] & \sV \ar[d]^{f} \\  
\sU \ar@{^{(}->}[r]^{e} & \sX,}
\end{equation}
where $f$ is an {\'e}tale morphism (not necessarily representable)
and $e$ is an open immersion with reduced
complement $\sZ$ such that the induced map $\sZ \times_\sX \sV \to \sZ$ is an isomorphism.
 Nisnevich squares form a cd-structure on the category of stacks, in the sense of \cite{Voev}.

\begin{prop}\label{prop:Nis-cd}
Let $f\colon \sY\to\sX$ be a Nisnevich covering. Then there exist sequences of quasi-compact open substacks 
\[
\sY_1\subset\dotsb\subset\sY_n\subset\sY,\quad \emptyset=\sX_0\subset\sX_1\subset\dotsb\subset\sX_n=\sX,
\]
such that $f(\sY_i)\subset \sX_i$ and such that each square
\begin{equation*}
\xymatrix@C1pc{
\sX_{i-1}\times_\sX\sY_i \ar@{^{(}->}[r] \ar[d] & \sY_i \ar[d]^{f} \\  
\sX_{i-1} \ar@{^{(}->}[r] & \sX_i,}
\end{equation*}
is a Nisnevich square.
\end{prop}

\begin{proof}
	The proof is exactly the same as \cite[Proposition 1.4]{MV}.
	Let $\sX_0\subset\dotsb\subset\sX_n$ be a monomorphic splitting sequence for $f$ (\propref{prop:splitting-sequence}), 
	and let $s_i:\sX_i\setminus\sX_{i-1}\to\sY\times_\sX(\sX_i\setminus\sX_{i-1})$
	be a monomorphic section of the projection. Then $s_i$ is an open immersion, so the complement 
	of the image of $s_i$ is a closed substack $\sZ_i\subset\sY\times_\sX\sX_i$.
	We can then take $\sY_i=(\sY\times_\sX\sX_i)\setminus\sZ_i$.
\end{proof}

\propref{prop:Nis-cd} implies that the Grothendieck topology
associated with the Nisnevich cd-structure is exactly the topology generated
by Nisnevich coverings.
The Nisnevich cd-structure on the category
of stacks clearly satisfies the assumptions of 
Voevodsky's descent criterion \cite[Theorem 3.2.5]{AHW}.
It follows that a presheaf of spaces or spectra $\sF$ satisfies descent
for Nisnevich coverings if and only if,
for every Nisnevich square~\eqref{eqn:Nis-square}, the induced square
\begin{equation*}
\xymatrix@C1pc{
\sF(\sX) \ar[r]^-{e^*} \ar[d]_{f^*} & \sF(\sU) \ar[d] \\
\sF(\sV) \ar[r] & \sF(\sW)}
\end{equation*}
is homotopy Cartesian.

The following recent result of Alper, Hall and Rydh \cite{AHR} on the
Nisnevich-local structure of some stacks
 will play an important role in the proof of our cdh-descent
theorem.
 
\begin{thm}[Alper--Hall--Rydh]\label{thm:Nis-local-aff}
Let $\sX$ be a stack, let $x\in\sX$ be a point, and let $\eta_x$ be its
residual gerbe.
Suppose that the stabilizer of $\sX$ at a representative of $x$ is a 
linearly reductive almost multiplicative group scheme.
Then there exists:
\begin{itemize}
	\item a morphism of affine schemes $U\to S$,
	\item a linearly reductive almost multiplicative group scheme $G$ over $S$ acting on $U$,
	\item a commutative diagram of stacks
\begin{equation*}
\xymatrix@C1pc{
\eta_x \ar@{^{(}->}[r] \ar@{=}[d] & [U/G] \ar[d]^{f} \\
\eta_x \ar@{^{(}->}[r] & \sX,}
\end{equation*}
where $f$ is {\'e}tale.
\end{itemize}
If $\sX$ has affine diagonal, we can moreover choose $f$ affine.
If $\sX$ has finite inertia and coarse moduli space $\pi:\sX\to X$,
 we can take $S$ to be an étale neighborhood of $\pi(x)$ in $X$.
\end{thm}

\begin{rem}
	Linearly reductive almost multiplicative group schemes are called \emph{nice}
	in \cite{HR-1} and \cite{AHR}, but this terminology is used differently in \cite{KR-2}, so we will avoid using it.
\end{rem}

\section{Perfect complexes on algebraic stacks}
\label{sec:perfect-complexes}

\subsection{Sheaves on stacks}\label{sec:sheaves}
Let $\sX$ be a stack. 
Let $\LisEt(\sX)$ denote the {\it lisse-\'etale site} of $\sX$.
Its objects are smooth morphisms $X \to \sX$, where $X$ is a quasi-compact 
quasi-separated scheme. The coverings are generated by the {\'e}tale covers of 
schemes.
Let $\Mod(\sX)$ (resp.\ $\QC(\sX)$) 
denote the abelian category of sheaves of $\sO_{\sX}$-modules (resp.\ of quasi-coherent sheaves) on 
$\LisEt(\sX)$.
It is well known that $\QC(\sX)$ 
and $\Mod(\sX)$ are Grothendieck abelian categories and hence have enough injectives
and all limits.

Let $\Ch(\sX)$ (resp.\ $\Ch_{\qc}(\sX)$) denote the 
category of all (possibly unbounded) chain complexes over $\Mod(\sX)$ 
(resp.\ the full subcategory of $\Ch(\sX)$ consisting of those chain complexes 
whose cohomology lies in $\QC(\sX)$). 
Let $D(\sX)$ and $D_{\qc}(\sX)$ denote their corresponding 
 derived categories, obtained by inverting quasi-isomorphisms.
If $\sZ \inj \sX$ is a closed substack with open complement $j: \sU \inj \sX$, 
we let 
$$\Ch_{\qc,\sZ}(\sX) = \{ \sF\in\Ch_{\qc}(\sX)\: |\: j^*(\sF) \text{ is quasi-isomorphic to } 0\}.$$ 
The derived category of $\Ch_{\qc,\sZ}(\sX)$ will be denoted by $D_{\qc,\sZ}(\sX)$.

Let $j: \sX \to \sY$ be a smooth morphism of algebraic stacks.
We then have the pullback functor
$j^*: \Mod(\sY) \to \Mod(\sX)$ which preserves quasi-coherent sheaves.
Since $j$ is smooth, the functor $j^*$ is simply the restriction
functor under the inclusion $\LisEt(\sX)\subset\LisEt(\sY)$.

Recall from \cite[Definition~I.4.2]{SGA6} that a complex of
$\sO_X$-modules on a scheme $X$ is perfect if it is locally 
quasi-isomorphic to a bounded complex of locally free sheaves.

\begin{defn} \label{defn:Perfect-complex}
Let $\sX$ be a stack.
A chain complex $P \in \Ch_{\qc}(\sX)$ is called {\it perfect} if for any 
affine scheme $U = \Spec(A)$ with a smooth morphism $s: U \to \sX$, 
the complex of $A$-modules $s^*(P) \in \Ch(\Mod(A))$ is quasi-isomorphic 
to a bounded complex of finitely generated projective $A$-modules.
Equivalently, $s^*(P)$ is a perfect complex in $\Ch(\Mod(A))$ in 
the sense of \cite{TT}.
\end{defn}

It follows from \cite[Lemma~2.5]{KR-2} that the above definition coincides
with that of \cite{TT} if $\sX$ is a scheme.
We shall denote the derived category of perfect complexes on $\sX$ by $D_{\perf}(\sX)$.
The derived category of perfect complexes on $\sX$ whose cohomology is supported on 
a closed substack $\sZ$ will be denoted by $D_{\perf,\sZ}(\sX)$.

We will also need to use the canonical dg-enhancements of the triangulated categories
$D_{\qc}(\sX)$ and $D_{\perf}(\sX)$, denoted by $\mathsf D_{\qc}(\sX)$ and $\mathsf D_{\perf}(\sX)$,
respectively, whose construction we now recall. 
 If $\sX$ is an affine scheme, $\mathsf D_{\qc}(\sX)$ is the usual symmetric monoidal
derived dg-category of $\sO(\sX)$. The 2-category of stacks embeds fully faithfully
in the 2-category of presheaves of groupoids on affine schemes, which further embeds in
the $\infty$-category $\mathrm{sPre}(\mathbf{Aff})$ of simplicial presheaves on affine schemes. 
Then one defines $\mathsf D_{\qc}$ as a presheaf of symmetric monoidal dg-categories
on $\mathrm{sPre}(\mathbf{Aff})$ to be the homotopy right Kan extension of $\mathsf D_{\qc}|_\mathbf{Aff}$
(see \cite[\S6.2]{SAG}). 
In other words, it is the unique extension of $\mathsf D_{\qc}|_\mathbf{Aff}$ that transforms homotopy
colimits into homotopy limits. One can show that $\mathsf D_{\qc}$ satisfies descent for
the fpqc topology on $\mathrm{sPre}(\mathbf{Aff})$ \cite[Proposition 6.2.3.1]{SAG}.
For $\sX$ a stack, the homotopy category of $\mathsf D_{\qc}(\sX)$ is then
equivalent to $D_{\qc}(\sX)$. 
If $\sX$ is an algebraic space (or more generally a Deligne--Mumford stack), this is proved in
\cite[Proposition 6.2.4.1]{SAG}. In general, this follows from the description of $D_{\qc}(\sX)$
in terms of a smooth representable cover of $\sX$ by an algebraic space, 
see for instance \cite[\S1.1]{HR-2}.
Finally, $\mathsf D_{\perf}\subset\mathsf D_{\qc}$ is the full symmetric monoidal
dg-subcategory spanned by the dualizable objects. Since the process of
passing to dualizable objects preserves homotopy limits of dg-categories
\cite[Proposition 4.6.1.11]{HA}, $\mathsf D_{\perf}$ is similarly
the unique extension of $\mathsf D_{\perf}|_\mathbf{Aff}$ to $\mathrm{sPre}(\mathbf{Aff})$ that transforms
homotopy colimits into homotopy limits, and it satisfies fpqc descent.

\begin{prop}\label{prop:Excision}
	Let $f:\sX'\to\sX$ be an étale morphism of stacks and let
	$\sZ \subset \sX$ be a closed substack with quasi-compact open complement such that the projection
	$\sZ\times_\sX\sX'\to\sZ$ is an isomorphism of associated reduced stacks.
	Then the functor \[f^*: D_{\perf, \sZ}(\sX) \to D_{\perf, \sZ \times_{\sX} \sX'}(\sX')\]
	is an equivalence of triangulated categories.
\end{prop}
\begin{proof}
The presheaf of dg-categories $\sX\mapsto \mathsf D_{\perf}(\sX)$
satisfies descent for the fpqc topology on stacks. In particular,
it satisfies Nisnevich descent, so that the square of dg-categories
\begin{equation*}
\xymatrix@C1pc{
\mathsf D_{\perf}(\sX) \ar[r] \ar[d]_{f^*} & \mathsf D_{\perf}(\sX\setminus\sZ) \ar[d] \\
\mathsf D_{\perf}(\sX') \ar[r] & \mathsf D_{\perf}(\sX'\setminus (\sZ \times_{\sX} \sX')).}
\end{equation*}
is homotopy Cartesian. It follows that $f^*$ induces an equivalence between the kernels
of the horizontal functors.
\end{proof}

\subsection{Perfect stacks}\label{sec:CPerf}

\begin{defn}\label{defn:perfect}
Let $\sX$ be a stack. We shall say that $\sX$ is {\sl perfect} if
the triangulated category $D_{\qc}(\sX)$ is compactly generated and $\sO_\sX$ is compact in $D_{\qc}(\sX)$.

If $\sZ\subset\sX$ is a closed substack with quasi-compact open complement, 
we shall say that the pair $(\sX,\sZ)$ is {\sl perfect}
if $\sX$ is perfect and there exists a perfect complex on $\sX$ with support $\lvert\sZ\rvert$.
\end{defn} 

We will see in \propref{prop:cpt-obj-D_(qc)(X on Z)} below that
 our notion of perfect stack agrees with the one introduced in \cite{BZFN},
except that we do not require perfect stacks to have affine diagonal.

Let $f: \sX' \to \sX$ be a morphism of
stacks. We say that $f$ is {\sl concentrated} if for every morphism
$g: \sZ \to \sX$, the morphism $f': \sX' \times_{\sX} \sZ \to \sZ$ has
finite cohomological dimension for quasi-coherent sheaves.

\begin{lem}\label{lem:Concentrated}
Let $f: \sX' \to \sX$ be a representable morphism of stacks.
Then $f$ is concentrated. 
In particular, if $\sO_\sX$ is compact, then $\sO_{\sX'}$ is compact.
\end{lem}
\begin{proof}
Since $f$ is representable, and since concentrated morphisms have
faithfully flat descent by \cite[Lemma~2.5 (2)]{HR-2}, we can assume that
$f$ is a morphism of algebraic spaces. Now, the result follows
because any quasi-compact and quasi-separated morphism of 
algebraic spaces is concentrated (see \cite[Tag 073G]{SP}). 
For the second statement, it suffices to show
using \cite[Theorem~5.1]{Neeman-1} that the right adjoint $f_*: D_{\qc}(\sX') \to D_{\qc}(\sX)$
of $f^*$ preserves  small coproducts.
 This follows from the first statement and \cite[Theorem 2.6 (3)]{HR-2}.
\end{proof} 

\begin{prop} \label{prop:cpt-obj-D_(qc)(X on Z)}
Let $(\sX,\sZ)$ be a perfect pair.
Then the triangulated category $D_{\qc,\sZ}(\sX)$ is compactly generated.
Moreover, an object of $D_{\qc,\sZ}(\sX)$
is compact if and only if it is perfect.
\end{prop}

\begin{proof}
Since $\sO_{\sX} \in D_{\qc}(\sX)$ is compact and since a perfect complex
on $\sX$ is dualizable, it follows that every perfect complex on
$\sX$ is compact. On the other hand, it follows from the proofs of
\cite[Proposition~2.7, Lemma~2.8]{KR-2} that compact objects of 
$D_{\qc}(\sX)$ and $D_{\qc,\sZ}(\sX)$ are perfect. The only remark we need to make 
here is that the proofs in {\sl loc.\ cit.}
assume that $\sX$ is a quotient stack.
However, this assumption is used only to ensure that  
if we choose an atlas $u: U \to \sX$, then $u$ has finite
cohomological dimension for quasi-coherent sheaves. But this follows from
\lemref{lem:Concentrated} because $\sX$ has representable diagonal and hence
$u$ is representable.
Finally, the existence of a perfect complex with support $\lvert\sZ\rvert$
implies, by \cite[Lemma 4.10]{HR-2}, that $D_{\qc,\sZ}(\sX)$ is compactly generated.
\end{proof}

\begin{lem}\label{lem:Quasi-proj}
Let $f: \sY \to \sX$
be a schematic morphism of stacks with a relatively ample family of line bundles. 
If $D_{\qc}(\sX)$ is compactly 
generated, so is $D_{\qc}(\sY)$.
\end{lem}

\begin{proof}
	Let $\{\sL_i\}_{i\in I}$ be an $f$-ample family of line bundles on $\sY$.
By \lemref{lem:Concentrated}, $f$ is a concentrated morphism.
It follows from \cite[Theorem 2.6 (3)]{HR-2} that $f_*:D_\qc(\sY)\to D_\qc(\sX)$
preserves small coproducts, and hence that its left adjoint $f^*$ preserves
compact objects. It will therefore suffice to show that $D_\qc(\sY)$ is generated by the objects
$f^*(\sF)\otimes \sL_i^{\otimes -n}$, for $\sF\in D_\qc(\sX)$ compact, $i\in I$, and $n\ge 1$.
So let $\sG \in D_{\qc}(\sY)$ be such that 
$\Hom(f^*(\sF) \otimes \sL_i^{\otimes -n},\sG) = 0$ for every $\sF$ compact, $i\in I$, and $n\ge 1$.
By adjunction, we have
$\Hom(\sF, f_*(\sG \otimes \sL_i^{\otimes n})) = 0$.
Since $D_{\qc}(\sX)$ is compactly generated, it follows that
\begin{equation}\label{eqn:Quasi-proj-0}
f_*(\sG \otimes \sL_i^{\otimes n}) = 0
\end{equation}
for every $i\in I$ and $n\ge 1$.

To show that $\sG = 0$ in $D_{\qc}(\sY)$, we let $u: U \to \sX$ be a smooth
surjective morphism such that $U$ is affine. This gives rise to a Cartesian
square
\begin{equation*}\label{eqn:Quasi-proj-1}
\xymatrix@C1pc{
V \ar[r]^-{v} \ar[d]_{g} & \sY \ar[d]^{f} \\
U \ar[r]^-u & \sX}
\end{equation*}
where $V$ is a scheme.
Since $v$ is faithfully flat, it suffices to show that
$v^*(\sG) = 0$. It follows from \cite[Corollary~4.13]{HR-2} and
~\eqref{eqn:Quasi-proj-0} that 
$g_*(v^*\sG \otimes v^*(\sL_i)^{\otimes n}) = 0$ for all 
$i\in I$. Replacing $\sY$ by $V$ and $\sL_i$ by $v^*(\sL_i)$, we
can assume that $\sX$ is an affine scheme, so that $\sY$ is a scheme
 and $\{\sL_i\}_{i\in I}$ is an ample family of line bundles on $\sY$.
In this case, \eqref{eqn:Quasi-proj-0} says that
$\Hom(\sL_i^{\otimes -n}[m], \sG) = 0$ for all $i\in I$, $n\ge 1$, and $m \in \Z$.
But this implies that $\sG$ is acyclic because
$D_{\qc}(\sY)$ is generated by $\{\sL_i^{\otimes -n}\}_{i\in I,n\ge 1}$.
Indeed, $D_\qc(\sY)$ is compactly generated by bounded complexes of vector bundles
\cite[Theorem 2.3.1 (d)]{TT}, and every vector bundle admits an 
epimorphism from a sum of line bundles of the form $\sL_i^{\otimes -n}$.
\end{proof}

\begin{prop}\label{prop:perfect}
Let $(\sX,\sZ)$ be a perfect pair.
\begin{enumerate}
	\item For every algebraic space $Y$ and closed subspace $W\subset Y$ with quasi-compact open complement, $(\sX\times Y,\sZ\times W)$ is perfect.
	\item For every schematic morphism $f\colon\sY\to\sX$ with a relatively ample family of line bundles, $(\sY,\sY\times_\sX\sZ)$ is perfect.
\end{enumerate}
\end{prop}

\begin{proof}
	Let $\sP$ be a perfect complex on $\sX$ with support $\lvert\sZ\rvert$.
	
	(1) By \cite[Theorem A]{HR-2}, there exists a perfect complex $\sQ$ on $Y$ with support $\lvert W\rvert$.
	Then $\pi_1^*(\sP)\otimes\pi_2^*(\sQ)$ is a perfect complex on $\sX\times Y$ with support $\lvert\sZ\times W\rvert$. Since the projection $\pi_1:\sX\times Y\to\sX$ is representable, $\sO_{\sX\times Y}$ is compact by \lemref{lem:Concentrated}.
	It remains to show that $D_{\qc}(\sX\times Y)$ is compactly generated. 
	We claim that there is an equivalence of presentable dg-categories
	\begin{equation}\label{eqn:DqcKunneth}
	\mathsf D_{\qc}(\sX\times Y)\simeq \mathsf D_{\qc}(\sX)\otimes \mathsf D_{\qc}(Y).
	\end{equation}
	Since the tensor product of compactly generated dg-categories is compactly generated, this will complete the proof. 
	Since $Y$ is a quasi-compact and quasi-separated algebraic space, the dg-category $\mathsf D_{\qc}(Y)$ is dualizable \cite[\S9.4]{SAG}, and hence tensoring with $\mathsf D_{\qc}(Y)$ preserves homotopy limits. Since $\mathsf D_{\qc}(-)$ is the homotopy right Kan extension of its restriction to affine schemes, we are reduced to proving~\eqref{eqn:DqcKunneth} when $\sX$ is an affine scheme, in which case it is a special case of \cite[Corollary 9.4.2.4]{SAG}.
	
	(2) The perfect complex $f^*(\sP)$ has support $\lvert\sY\times_\sX\sZ\rvert$.
	By \lemref{lem:Concentrated}, $\sO_\sY$ is compact. 
	It remains to show that $D_{\qc}(\sY)$ is compactly generated, but this follows from \lemref{lem:Quasi-proj}.
\end{proof}

\begin{prop}\label{prop:Localization}
	Let $(\sX,\sZ)$ be a perfect pair and let $j:\sU\inj\sX$ be the open immersion complement to $\sZ$. Then
	\[
	j^*: \frac{D_{\perf}(\sX)}{D_{\perf,\sZ}(\sX)} \to D_{\perf}(\sU)
	\]
	is an equivalence of triangulated categories, up to direct factors.
\end{prop}

\begin{proof}
	For any pair $(\sX,\sZ)$, we have an equivalence of triangulated categories
	\[
	j^*: \frac{D_{\qc}(\sX)}{D_{\qc,\sZ}(\sX)} \to D_{\qc}(\sU).
	\]
	Indeed, the functor $j_*: D_{\qc}(\sU)\to D_{\qc}(\sX)$ is fully faithful by flat base change,
	so $j_*j^*$ is a localization endofunctor of $D_{\qc}(\sX)$ whose kernel is $D_{\qc,\sZ}(\sX)$
	by definition. The claim now follows from \cite[Proposition 4.9.1]{Krause}.
	If $(\sX,\sZ)$ is perfect, then $\sU$ is also perfect by \propref{prop:perfect} (2).
	By \propref{prop:cpt-obj-D_(qc)(X on Z)}, all three categories are compactly generated and their subcategories
	of compact and perfect objects coincide. We conclude using \cite[Theorem 5.6.1]{Krause}.
\end{proof}

\begin{prop}\label{prop:continuity-perf}
	Suppose that $\sX$ is the limit of a filtered diagram $(\sX_\alpha)$ of perfect stacks with affine transition morphisms.
	Then $\sX$ is perfect and the canonical map
	\begin{equation}\label{eqn:limDperf}
	\hocolim_\alpha \mathsf D_\perf(\sX_\alpha) \to \mathsf D_\perf(\sX)
	\end{equation}
	is a weak equivalence of dg-categories.
\end{prop}

\begin{proof}
It follows from \propref{prop:perfect} (2) that $\sX$ is perfect.
By \propref{prop:cpt-obj-D_(qc)(X on Z)}, 
$\mathsf D_\qc(\sX)$ is compactly generated and $\mathsf D_\qc(\sX)^c=\mathsf D_\perf(\sX)$,
and similarly for each $\sX_\alpha$.
Since the pullback functors $\mathsf D_\qc(\sX_\alpha)\to\mathsf D_\qc(\sX_\beta)$
preserve compact objects, it follows from \cite[Propositions 5.5.7.6 and 5.5.7.8]{HTT} and
\cite[Lemma 7.3.5.10]{HA} that 
 \eqref{eqn:limDperf} is a weak equivalence if and only if the canonical map
 \begin{equation}\label{eqn:limDqc}
 	\mathsf D_\qc(\sX) \to \holim_\alpha \mathsf D_\qc(\sX_\alpha)
 \end{equation}
 is a weak equivalence.
 Choosing a smooth hypercover of some $\sX_\alpha$ by schemes
and using flat base change, we see that the map~\eqref{eqn:limDqc} is the homotopy limit
of a cosimplicial diagram of similar maps with $\sX_\alpha$ replaced by a scheme.
Hence, it suffices to prove that~\eqref{eqn:limDperf}
is a weak equivalence when $\sX_\alpha$ is a scheme, but this follows from \cite[Proposition 3.20]{TT}.
\end{proof}

We now state the following two results of Hall and Rydh,
which provide many examples of perfect stacks.

\begin{thm}[Hall--Rydh]\label{thm:HRydh-1}
	Let $\sX$ be a stack satisfying one of the following properties.
	\begin{enumerate}
		\item $\sX$ has characteristic zero.
		\item $\sX$ has linearly reductive almost multiplicative stabilizers.
		\item $\sX$ has finitely presented inertia and
		linearly reductive almost multiplicative stabilizers at points of
		positive characteristic.
	\end{enumerate}
	Then $\sO_\sX$ is compact in $D_\qc(\sX)$.
\end{thm}

\begin{proof}
	See \cite[Theorem 2.1]{HR-1}.
\end{proof}

\begin{thm}[Hall--Rydh]\label{thm:HRydh-ex}
Let $\sX$ be a stack satisfying the following properties.
\begin{enumerate}
\item
$\sO_\sX$ is compact in $D_\qc(\sX)$.
\item
There exists a faithfully flat, representable, separated and quasi-finite
morphism $f:\sX' \to \sX$ of finite presentation such that $\sX'$ has affine stabilizers and satisfies the
resolution property.
\end{enumerate}
Then, for every closed substack $\sZ\subset\sX$, the pair $(\sX,\sZ)$ is perfect.
\end{thm}
\begin{proof}
By \lemref{lem:Concentrated}, $\sO_{\sX'} = f^*(\sO_{\sX})$ is compact in $D_{\qc}(\sX')$.
Since $\sX'$ has affine stabilizers and satisfies the resolution property, it has affine diagonal
by \cite[Theorem 1.1]{Gross}.
Since moreover $\sO_{\sX'}$ is compact, it follows from
\cite[Proposition~8.4]{HR-2} that $\sX'$ is {\sl crisp}.
We now apply \cite[Theorem~C]{HR-2} to conclude that $\sX$ is also crisp.
By definition of crispness, this implies that $(\sX,\sZ)$ is perfect.
\end{proof}

\begin{cor}\label{cor:CP-exm}
Let $\sX$ be a quasi-DM stack with separated diagonal and linearly reductive stabilizers. 
Then, for every closed substack $\sZ\subset\sX$, the pair $(\sX,\sZ)$ is perfect.
\end{cor}
\begin{proof}
Recall that a quasi-DM stack is a stack whose diagonal is quasi-finite. 
It follows from \cite[Tag 06MC]{SP} that a stack $\sX$ is quasi-DM if
and only if there exists an affine scheme $X$ and a faithfully flat map 
$f:X \to \sX$ of finite presentation which is quasi-finite.
Since the diagonal of $\sX$ is representable and separated, it follows that
$f$ is representable and separated.
Since $X$ is affine and hence has the resolution property,
the corollary follows from Theorems \ref{thm:HRydh-1} and \ref{thm:HRydh-ex}. 
\end{proof}

\begin{cor}\label{cor:Nis-local-aff-1}
Let $\sX$ be a stack with affine diagonal and linearly reductive almost multiplicative stabilizers.
Then, for every closed substack $\sZ\subset\sX$, $(\sX,\sZ)$ is perfect.
\end{cor}

\begin{proof}
By \thmref{thm:Nis-local-aff}, there exists a Nisnevich covering $\{f_i\colon [U_i/G_i]\to\sX\}_{i\in I}$ where
 $f_i$ is affine, $U_i$ is affine over an affine scheme $S_i$, 
 and $G_i$ is a linearly reductive almost multiplicative group scheme over $S_i$.
 By taking a further affine Nisnevich covering of $S_i$, we can ensure that $G_i$
 is almost isotrivial and hence that $[U_i/G_i]$ has the
 resolution property (see \cite[Example 2.8 and Remark 2.9]{Hoyois-1}).
 By~\corref{cor:Nis-qc}, we can also assume that $I$ is finite.
Let $\sX'=\coprod_i[U_i/G_i]$. Then the induced map $\sX'\to\sX$ is faithfully flat, quasi-finite, and affine.
Since $\sX'$ has the resolution property, we conclude that $(\sX,\sZ)$ 
is perfect by Theorems \ref{thm:HRydh-1} and \ref{thm:HRydh-ex}.
\end{proof}

\section{$K$-theory of perfect stacks}\label{sec:Nis-desc}
In this section, we establish some descent properties
of the $K$-theory, negative $K$-theory, and homotopy $K$-theory
of stacks. Special cases of these
results were earlier proven in \cite{KO}, \cite{KR-2} and \cite{Hoyois-2}.

\subsection{Localization, excision, and continuity}\label{sec:LMV}
Let $\sX$ be an algebraic stack. The algebraic $K$-theory spectrum of $\sX$ 
is defined to be the $K$-theory spectrum of the
complicial biWaldhausen category of perfect complexes in $\Ch_{\qc}(\sX)$
in the sense of \cite[\S1.5.2]{TT}. Here, the complicial biWaldhausen category
structure is given with respect to the degree-wise split monomorphisms as 
cofibrations and quasi-isomorphisms as weak equivalences.
This $K$-theory spectrum is denoted by $K(\sX)$.
Equivalently, one may define $K(\sX)$ as the $K$-theory spectrum
of the dg-category $\mathsf D_{\perf}(\sX)$ (see \cite[Corollary 7.12]{BGT}).
Note that the negative homotopy groups of $K(\sX)$ are zero (see \cite[\S1.5.3]{TT}).
We shall extend this definition to negative integers in the next section.
For a closed substack $\sZ$ of $\sX$, 
$K(\sX ~ {\rm on} ~ \sZ)$ is the $K$-theory spectrum of the complicial 
biWaldhausen category of those perfect complexes on $\sX$ which are acyclic
on $\sX \setminus \sZ$.

\begin{thm}[\bf Localization]\label{thm:Localization}
Let $(\sX,\sZ)$ be a perfect pair and let $j:\sU \inj \sX$ be the open
immersion complement to $\sZ$.
Then the morphisms of spectra
$K(\sX ~ {\rm on} ~ \sZ) \to K(\sX) \xrightarrow{j^*} K(\sU)$
induce a long exact sequence
\begin{align*}
\cdots \to K_i(\sX ~ {\rm on} ~ \sZ) \to K_i(\sX) & \to K_i(\sU) \to 
K_{i-1}(\sX ~ {\rm on} ~ \sZ) \to \cdots \\
& \to K_0(\sX ~ {\rm on} ~ \sZ) \to K_0(\sX) \to K_0(\sU).
\end{align*}
\end{thm}

\begin{proof}
	This follows from~\propref{prop:Localization} as in \cite[Theorem~3.4]{KR-2}.
\end{proof}

\begin{thm}[\bf Excision]\label{thm:Excision}
Let $f:\sX'\to\sX$ be an étale morphism of stacks and let
$\sZ \subset \sX$ be a closed substack with quasi-compact open complement such that the projection
$\sZ\times_\sX\sX'\to\sZ$ is an isomorphism of associated reduced stacks.
Then the map $f^*: K(\sX \ {\rm on} \ \sZ) 
\to K(\sX' \ {\rm on} \ \sZ \times_{\sX}\sX')$
is a homotopy equivalence of spectra.
\end{thm}
\begin{proof}
This follows from~\propref{prop:Excision} using \cite[Theorem~1.9.8]{TT}.
\end{proof}

\begin{thm}[\bf Continuity]\label{thm:continuity}
	Let $\sX$ be the limit of a filtered diagram $(\sX_\alpha)$ of perfect stacks
	with affine transition morphisms. Then the canonical map
	\[
	\hocolim_\alpha K(\sX_\alpha)\to K(\sX)
	\]
	is a homotopy equivalence.
\end{thm}

\begin{proof}
	This follows from \propref{prop:continuity-perf} and the fact that $K$ preserves
	filtered homotopy colimits of dg-categories.
\end{proof}

\subsection{The Bass construction and negative $K$-theory}\label{sec:Neg-K}
The non-connective $K$-theory spectrum of any stack may be defined
from the complicial biWaldhausen category of perfect complexes,
following Schlichting \cite{Schl}, or from the dg-category
$\mathsf D_{\perf}(\sX)$, following Cisinski--Tabuada \cite{CT}.
This allows one to define the negative $K$-theory of stacks.

In this subsection, we will see that for perfect stacks a non-connective
$K$-theory spectrum $K^B$ can be defined much more explicitly using
the construction of Bass--Thomason--Trobaugh. One may prove that this
construction agrees with those of Schlichting and Cisinski--Tabuada
exactly as in \cite[Theorem 3.21]{KR-2}.

The $K^B$-theory spectrum $K^B(\sX)$
was constructed in \cite[\S3E]{KR-2} based on
the following two assumptions.
\begin{enumerate}
\item
$\sX$ is a quotient stack of the form $[X/G]$ over a field, 
where $G$ is a linearly reductive group scheme.
\item
$\sX$ satisfies the resolution property.
\end{enumerate}
Since perfect stacks need not satisfy these conditions, we cannot
 directly quote the results of \cite{KR-2} for the construction of
the $K^B$-theory of stacks.
But the proofs are identical to those in
\cite{KR-2} using Theorems~\ref{thm:Localization} and
~\ref{thm:Excision}, so we shall only give a brief sketch of the construction.
The existence of the $K^B$-theory is based on the following version of
the fundamental theorem of Bass.

\begin{thm} [\bf{Bass fundamental theorem}] \label{thm:Bass-fund}
Let $\sX$ be a perfect stack and let 
$\sX[T]$ denote the stack 
$\sX \times \Spec(\Z[T])$. Then the following hold.
\begin{enumerate}
\item For $n \ge 1$, there is an exact sequence 
\begin{align*} 
0 \to K_n(\sX) \xrightarrow{(p_1^*, -p_2^*)}  & K_n(\sX[T]) \oplus 
K_n(\sX[T^{-1}]) \\ &  \xrightarrow{(j_1^*, j_2^*)}  K_n(\sX[T, T^{-1}])
 \xrightarrow{\partial_T} K_{n-1}(\sX) \to 0.
\end{align*}
Here $p_1^*, p_2^*$ are induced by the projections $\sX[T] \to \sX$, etc.
and $j_1^*, j_2^*$ are induced by the open immersions 
$\sX[T^{\pm1}] = \sX[T,T^{-1}] \to \sX[T]$,
etc. The sum of these exact sequences for  $n = 1, 2, \cdots$ 
is an exact sequence of graded $K_*(\sX)$-modules.\
\item For $n \ge 0$, $\partial_T: K_{n+1}(\sX[T^{\pm1}]) \to K_n(\sX)$
is naturally split by a map $h_T$ of $K_*(\sX)$-modules. 
Indeed, the cup product with $T \in K_1(\Z[T^{\pm1}])$ splits $\partial_T$ up to
a natural automorphism of $K_n(\sX)$.\
\item There is an exact sequence
\[
0 \to K_0(\sX) \xrightarrow{(p_1^*, -p_2^*)}  K_0(\sX[T]) \oplus 
K_0(\sX[T^{-1}])  \xrightarrow{(j_1^*, j_2^*)} K_0(\sX[T^{\pm1}]).
\]
\end{enumerate}
\end{thm}
\begin{proof}
The proof of this theorem is word by word identical to the proof of
its scheme version given in \cite[Theorem~6.1]{TT}, once we know that
the algebraic $K$-theory spectrum satisfies the following properties:
\begin{enumerate}
\item
The projective bundle formula for the projective line $\P^1_\sX$.
\item
Localization for the pairs $(\P^1_\sX,\sX[T^{-1}])$ and $(\sX[T],\sX[T^{\pm 1}])$.
\item
Excision.
\end{enumerate}
Property (1) follows from \cite[Theorem~3.8]{KR-2} which holds for any algebraic
stack. Property (2) follows from~\propref{prop:perfect} (1) and~\thmref{thm:Localization},
and property (3) is~\thmref{thm:Excision}.
\end{proof}

As an immediate consequence of \thmref{thm:Bass-fund}, one obtains the
following.

\begin{thm}\label{thm:K-G-main}
Let $\sX$ be a perfect stack. Then there is a 
spectrum
$K^B(\sX)$ together with a natural map $K(\sX) \to K^B(\sX)$ of 
spectra inducing isomorphisms 
$\pi_iK(\sX) \cong \pi_iK^B(\sX)$ for $i \ge 0$, 
which satisfies the following properties.
\begin{enumerate}
\item
Let $\sZ \subset \sX$ be a closed substack
with quasi-compact open complement $j: \sU \inj \sX$ such that $(\sX,\sZ)$ is perfect. 
Then there is a homotopy fiber sequence of spectra
\[
K^B(\sX \ {\rm on} \ \sZ) \to  K^B(\sX) \xrightarrow{j^*} K^B(\sU).
\]
\item
Let $f: \sY \to \sX$ be an {\'e}tale map between perfect stacks such that the projection 
$\sZ \times_{\sX} \sY \to \sZ$ is an isomorphism on the associated reduced
stacks.
Then the map $f^*: K^B(\sX \ {\rm on} \ \sZ)  \to 
K^B(\sY \ {\rm on} \ \sZ \times_{\sX} \sY)$ is a homotopy equivalence.
\item
Let $\pi: \P(\sE) \to \sX$ be the projective bundle 
associated to a vector bundle $\sE$ on $\sX$ of rank $r$.
Then the map
\[
\prod\limits_{0}^{r-1} K^B(\sX) \to  K^B(\P(\sE))
\]
that sends $(a_0, \cdots , a_{r-1})$ to $\sum_i
\sO(-i) \otimes \pi^*(a_i)$ is a homotopy equivalence.
\item
Let $i:\sY\inj\sX$ be 
a regular closed immersion and let $p:\sX'\to\sX$ be the blow-up of $\sX$ with center $\sY$.
Then the square of spectra
\begin{equation*}
\xymatrix@C1pc{
K^B(\sX) \ar[r]^-{i^*} \ar[d]_{p^*} & K^B(\sY) \ar[d] \\
K^B(\sX') \ar[r] & K^B(\sX'\times_\sX\sY)}
\end{equation*}
is homotopy Cartesian.
\item Suppose that $\sX$ is the limit of a filtered diagram $(\sX_\alpha)$ of perfect stacks
with affine transition morphisms. Then the canonical map
\[
\hocolim_\alpha K^B(\sX_\alpha)\to K^B(\sX)
\]
is a homotopy equivalence.
\end{enumerate}
\end{thm}

\begin{proof}
The spectrum $K^B(\sX)$ is constructed
word by word using \thmref{thm:Bass-fund} and 
the formalism given in (6.2)--(6.4) of \cite{TT} for the 
case of schemes. The proof of the asserted properties is a standard deduction
from the analogous properties of $K(\sX)$. The sketch of this deduction for (1)--(4)
can be found in \cite[Theorem~3.20]{KR-2}.
Note that quasi-compact open substacks of $\sX$, projective bundles over $\sX$, and blow-ups of $\sX$
are perfect stacks by~\propref{prop:perfect} (2).
For (5), it suffices to check that $\colim_\alpha \pi_nK^B(\sX_\alpha)\cong \pi_nK^B(\sX)$ for all $n\in\Z$.
This follows from \thmref{thm:continuity} since
$\pi_{-n}K^B(\sX)$, for $n>0$, is a natural retract of $K_0(\G_m^n\times\sX)$.
\end{proof}

\begin{cor}\label{cor:Nis-descent-KB}
Let
\begin{equation*}
\xymatrix{
\sW \ar@{^{(}->}[r] \ar[d] & \sV \ar[d]^{f} \\  
\sU \ar@{^{(}->}[r]^{e} & \sX,}
\end{equation*}
be a Nisnevich square of stacks, and suppose that the pairs $(\sX,\sX\setminus\sU)$
and $(\sV,\sV\setminus\sW)$ are perfect. Then the induced square of spectra
 \begin{equation*}
 	\xymatrix@C1pc{
 	K^B(\sX) \ar[r]^-{f^*} \ar[d]_{e^*} & K^B(\sV) \ar[d] \\
 	K^B(\sU) \ar[r] & K^B(\sW)}
 \end{equation*}
is homotopy Cartesian.
\end{cor}

\begin{proof}
	This follows immediately from~\thmref{thm:K-G-main} (1) and (2).
\end{proof}

\begin{remk}\label{remk:Rep-etale}
We remark that if $(\sX, \sX \setminus \sU)$ is a perfect pair and
if the map $f: \sV \to \sX$ in \corref{cor:Nis-descent-KB} is 
representable and separated, then $(\sV, \sV \setminus \sW)$ is automatically 
a perfect pair. The reason is that in this case, 
$f$ is quasi-affine by Zariski's Main Theorem for stacks 
\cite[Theorem~16.5]{LM} and one can apply 
\propref{prop:perfect} (2). 
\end{remk}

Since the homotopy groups of the two spectra $K(\sX)$ and $K^B(\sX)$ agree
in non-negative degrees by \thmref{thm:K-G-main}, we make the following
definition.

\begin{defn}\label{defn:Neg-K-def}
Let $\sX$ be a perfect stack and $i\in \Z$.
We let $K_i(\sX)$ denote the $i$-th
homotopy group of the spectrum $K^B(\sX)$.
\end{defn}

\subsection{The homotopy $K$-theory of perfect stacks}\label{sec:HKT}
For $n \in \N$, let 
\[
\Delta^n = \Spec\left(
\frac{\Z[t_0, \cdots , t_n]}{(\sum_i t_i - 1)}\right).
\]
Recall that $\Delta^{\bullet}$
is a cosimplicial scheme. For a perfect stack $\sX$,
the {\sl homotopy $K$-theory} of $\sX$ is defined as
\begin{equation*}
KH(\sX) = \hocolim_{n\in\Delta^{\rm op}} K^B(\sX \times \Delta^n).
\end{equation*}
There is a natural map $K^B(\sX) \to KH(\sX)$ induced by $0\in \Delta^{\rm op}$. 

\begin{thm}\label{thm:KH-Inv}
Let $\sX$ be a perfect stack.
\begin{enumerate}
\item
Let $\sZ \subset \sX$ be a closed substack
with quasi-compact open complement $j: \sU \inj \sX$ such that $(\sX,\sZ)$ is perfect. 
Then there is a homotopy fiber sequence of spectra
\[
KH(\sX \ {\rm on} \ \sZ) \to  KH(\sX) \xrightarrow{j^*} KH(\sU).
\]
\item
Let $f: \sY \to \sX$ be an {\'e}tale map between perfect stacks such that the projection 
$\sZ \times_{\sX} \sY \to \sZ$ is an isomorphism on the associated reduced
stacks.
Then the map $f^*: KH(\sX \ {\rm on} \ \sZ)  \to 
KH(\sY \ {\rm on} \ \sZ \times_{\sX} \sY)$ is a homotopy equivalence.
\item
Let $\pi: \P(\sE) \to \sX$ be the projective bundle 
associated to a vector bundle $\sE$ on $\sX$ of rank $r$.
Then the map
\[
\prod\limits_{0}^{r-1} KH(\sX) \to  KH(\P(\sE))
\]
that sends $(a_0, \cdots , a_{r-1})$ to $\sum_i
\sO(-i) \otimes \pi^*(a_i)$ is a homotopy equivalence.
\item
Let $i:\sY\inj\sX$ be 
a regular closed immersion and and let $p:\sX'\to\sX$ be the blow-up of $\sX$ with center $\sY$.
Then the square of spectra
\begin{equation*}
\xymatrix@C1pc{
KH(\sX) \ar[r]^-{i^*} \ar[d]_{p^*} & KH(\sY) \ar[d] \\
KH(\sX') \ar[r] & KH(\sX'\times_\sX\sY)}
\end{equation*}
is homotopy Cartesian.
\item Suppose that $\sX$ is the limit of a filtered diagram $(\sX_\alpha)$ of perfect stacks
with affine transition morphisms. Then the canonical map
\[
\hocolim_\alpha KH(\sX_\alpha)\to KH(\sX)
\]
is a homotopy equivalence.
\item
Let $u: \sE\to\sX$ be a vector bundle over $\sX$.
Then the induced map $u^*: KH(\sX) \to KH(\sE)$ is a 
homotopy equivalence.
\end{enumerate}
\end{thm}
\begin{proof}
Properties (1), (2), (3), (4), and (5) follow immediately from the
definition of $KH(\sX)$ and \thmref{thm:K-G-main}.
The proof of (6) for quotient stacks is given in \cite[Theorem~5.2]{KR-2}
and the same proof is valid for perfect stacks.
\end{proof}

\begin{cor}\label{cor:Nis-descent-KH}
Let
\begin{equation*}
\xymatrix{
\sW \ar@{^{(}->}[r] \ar[d] & \sV \ar[d]^{f} \\  
\sU \ar@{^{(}->}[r]^{e} & \sX,}
\end{equation*}
be a Nisnevich square of stacks, and suppose that the pairs $(\sX,\sX\setminus\sU)$
and $(\sV,\sV\setminus\sW)$ are perfect. Then the induced square of spectra
 \begin{equation*}
 	\xymatrix@C1pc{
 	KH(\sX) \ar[r]^-{f^*} \ar[d]_{e^*} & KH(\sV) \ar[d] \\
 	KH(\sU) \ar[r] & KH(\sW)}
 \end{equation*}
is homotopy Cartesian.
\end{cor}

\begin{proof}
	This follows immediately from~\thmref{thm:KH-Inv} (1) and (2).
\end{proof}

\begin{rem}
In \cite{Hoyois-2}, a potentially different definition of $KH$ is given for certain quotient stacks,
which forces $KH$ to be invariant with respect to vector bundle \emph{torsors} and not just vector bundles.
The two definitions agree for quotients of schemes by finite or multiplicative type groups,
as we will show in~\lemref{lem:torsor-inv}, but they may differ in general.
We do not know if the above definition of $KH$ has good properties for general perfect stacks.
\end{rem}

\section{$G$-theory and the case of regular stacks}\label{sec:Reg}

Our goal in this section is to show that perfect stacks 
that are Noetherian and regular have no negative $K$-groups.
We will do this by comparing the $K$-theory and $G$-theory of such stacks.

Let $\sX$ be a stack. Recall that $\Mod(\sX)$ is the abelian category of 
$\sO_{\sX}$-modules on the lisse-{\'e}tale site of $\sX$
 and $\QC(\sX)\subset  \Mod(\sX)$ is the abelian subcategory of 
quasi-coherent sheaves.

\begin{lem}\label{lem:Lurie*}
Assume that $\sX$ is a Noetherian stack.
Then the inclusion $\iota_{\sX}: \QC(\sX) \inj \Mod(\sX)$ induces an 
equivalence of the derived categories $D^{+}(\QC(\sX)) \xrightarrow{\simeq}
D^{+}_{\qc}(\sX)$.
\end{lem}
\begin{proof}
To show that $D^{+}(\QC(\sX)) \to D^{+}_{\qc}(\sX)$ is
full and faithful, it suffices, using standard reduction, to show that
the natural map ${\rm Ext}^i_{\QC(\sX)}(N,M) \to {\rm Ext}^i_{\Mod(\sX)}(N,M)$
is an isomorphism for all $i \in \Z$ for $N, M \in \QC(\sX)$.
Since this is clearly true for $i \le 0$, and since
$\iota_{\sX}: \QC(\sX) \to \Mod(\sX)$ is exact, it suffices to show that
this functor preserves injective objects.

Let $\sF$ be an injective quasi-coherent sheaf on $\sX$. Since a direct
summand of an injective object in $\Mod(\sX)$ is injective and, since
a quasi-coherent sheaf which is injective as a sheaf of $\sO_{\sX}$-modules
is also an injective quasi-coherent sheaf, it suffices to show that 
there is an inclusion $\sF \inj \sG$ in $\QC(\sX)$ such that 
$\iota_{\sX}(\sG)$ is injective in $\Mod(\sX)$.

Since $\sX$ is 
Noetherian, we can find a smooth atlas $u: U \to \sX$, where
$U$ is a Noetherian scheme. We can now find an inclusion
$u^*(\sF) \inj \sH$ in $\QC(U)$ such that $\sH$ is injective as a
sheaf of $\sO_U$-modules, by \cite[B.4]{TT}. We now consider the maps
\begin{equation}\label{eqn:Lurie*-0}
\sF \to u_*u^*(\sF) \to u_*(\sH).
\end{equation}
As $U$ is Noetherian, it is clear that $u_*(\sH)$ is a quasi-coherent sheaf on 
$\sX$.
Furthermore, $u_*$ has a left adjoint $u^*: \Mod(\sX) \to \Mod(U)$
which preserves quasi-coherent sheaves. Since $(u: U \to \sX)$ is an object
of $\LisEt(\sX)$, it follows that 
$u^*: \Mod(\sX) \to \Mod(U)$ is exact. In particular,
$u_*: \Mod(U) \to \Mod(\sX)$ has an exact left adjoint.
This implies that it must preserve injective sheaves. It follows that
$u_*(\sH)$ is a quasi-coherent sheaf on $\sX$ which is injective as
a sheaf of $\sO_{\sX}$-modules. 

Letting $\sG = u_*(\sH)$,
we are now left with showing that the two maps in ~\eqref{eqn:Lurie*-0}
are injective. The first map is injective because $u$ is faithfully flat and
$\sF$ is quasi-coherent. The second map is injective because
$u_*: \QC(U) \to \QC(\sX)$ is left exact and hence preserves injections.

To show that the functor $D^{+}(\QC(\sX)) \to D^{+}_{\qc}(\sX)$ is 
essentially surjective, we can use its full and faithfulness shown above
and an induction on the length to first see that 
$D^{b}(\QC(\sX)) \xrightarrow{\simeq} D^{b}_{\qc}(\sX)$.
Since every object of $D^+_{\qc}(\sX)$ is a colimit of objects
in $D^{b}_{\qc}(\sX)$ (using good truncations), 
a limit argument concludes the proof.
\end{proof}

\begin{lem}\label{lem:compact-gen*}
Let $\sX$ be as in \lemref{lem:Lurie*} and let $P \in  D(\QC(\sX))$ be a compact
object. Then the following hold.
\begin{enumerate}
\item
There exists an integer $r \ge 0$ such that $\Hom_{D_{\qc}(\sX)}(P, N[i]) = 0$
for all $N \in \QC(\sX)$ and $i > r$.
\item
There exists an integer $r \ge 0$ such that the natural map
\[
\tau^{\ge j}{\rm RHom}_{D_{\qc}(\sX)}(P, M) \to
\tau^{\ge j}{\rm RHom}_{D_{\qc}(\sX)}(P, \tau^{\ge j-r} M)
\]
is a quasi-isomorphism for all $M \in D_{\qc}(\sX)$ and integers $j$.
\item
There exists an integer $r \ge 0$ such that the natural map
\[
\tau^{\ge j}{\rm RHom}_{D(\QC(\sX))}(P, M) \to
\tau^{\ge j}{\rm RHom}_{D(\QC(\sX))}(P, \tau^{\ge j-r} M)
\] 
is a quasi-isomorphism for all $M \in D(\QC(\sX))$ and integers $j$.
\end{enumerate}
\end{lem}
\begin{proof}
It follows from \lemref{lem:Lurie*} that $\iota_{\sX}$ induces an equivalence
between the derived categories of perfect complexes of quasi-coherent
sheaves and perfect complexes of $\sO_{\sX}$-modules.
Since the compact objects of $D_{\qc}(\sX)$ are perfect 
\cite[Proposition~2.7]{KR-2}, it follows 
that $D(\QC(\sX))$ and $D_{\qc}(\sX)$ have equivalent full subcategories
of compact objects.
The parts (1) and (2) now follow from \cite[Lemma~4.5]{HR-2} and the 
proof of \cite[Lemma~2.4]{HNR} shows that (1) implies (3) for any
stack.
\end{proof}

\begin{lem}\label{lem:Lurie-unbounded}
Let $\sX$ be a Noetherian stack such that $D_{\qc}(\sX)$ is compactly 
generated.
Then $\iota_{\sX}: \QC(\sX) \to \Mod(\sX)$ induces 
an equivalence of the unbounded 
derived categories $D(\QC(\sX)) \xrightarrow{\simeq} 
D_{\qc}(\sX)$.
\end{lem}
\begin{proof}
Let $\Psi: D(\QC(\sX)) \to D_{\qc}(\sX)$ denote the derived functor
induced by $\iota_{\sX}$. 
We have shown in the proof of \lemref{lem:compact-gen*} that $\Psi$
restricts to an equivalence between the full subcategories
of compact objects. Using \cite[Lemma~4.5]{BIK}, it suffices
therefore to show that $D(\QC(\sX))$ is compactly generated.

So let $M \in D(\QC(\sX))$ be such that $\Hom_{D(\QC(\sX))}(P, M) = 0$
for every compact object $P$. We need to show that $M = 0$.
Since any compact object of $D(\QC(\sX))$ is 
perfect, and $\Psi$ is conservative and induces equivalence of compact
objects, it suffices to show that ${\rm RHom}(P, M) \xrightarrow{\simeq}
{\rm RHom}(\Psi(P), \Psi(M))$ for every perfect complex $P$.
Equivalently, we need to show that for every integer $j$, 
the map $\tau^{\ge j}{\rm RHom}(P, M) \to 
\tau^{\ge j}{\rm RHom}(\Psi(P), \Psi(M))$ is a quasi-isomorphism.
\lemref{lem:compact-gen*} now allows us to assume that
$M \in D^+(\QC(\sX))$. But in this case, the result follows from
\lemref{lem:Lurie*}.
\end{proof}

For a Noetherian stack $\sX$, let $G^{\rm naive}(\sX)$ denote the $K$-theory
spectrum of the exact category of coherent $\sO_{\sX}$-modules in the sense
of Quillen, and let $G(\sX)$ be the $K$-theory spectrum of the
complicial biWaldhausen category of cohomologically bounded pseudo-coherent 
complexes in $\Ch_{\qc}(\sX)$, in the sense of \cite[\S1.5.2]{TT}.
We have a natural map of spectra $G^{\rm naive}(\sX) \to G(\sX)$.

\begin{lem}\label{lem:G-thry}
Let $\sX$ be a Noetherian stack such that $D_{\qc}(\sX)$ is compactly generated.
Then the map $G^{\rm naive}(\sX) \to G(\sX)$ is a homotopy equivalence.
\end{lem}
\begin{proof}
It follows from 
Lemma~\ref{lem:Lurie-unbounded} that $G(\sX)$ is homotopy equivalent to
the $K$-theory of the Waldhausen category $\Ch_\pc(\QC(\sX))$ 
of cohomologically bounded pseudo-coherent chain complexes of quasi-coherent sheaves
on $\sX$. Let $\Ch^b(\Coh(\sX))$ denote the Waldhausen category of bounded complexes of
coherent $\sO_{\sX}$-modules.

Using the fact that every quasi-coherent sheaf on $\sX$ is a filtered
colimit of coherent subsheaves \cite[Proposition~15.4]{LM}, 
we can mimic the proof of \cite[Lemma~3.12]{TT} to conclude that
the inclusion $\Ch^b(\Coh(\sX)) \inj \Ch_\pc(\QC(\sX))$ induces
a homotopy equivalence between the associated $K$-theory spectra.
Since $G^{\rm naive}(\sX) \to K(\Ch^b(\Coh(\sX)))$ is also a homotopy
equivalence, by induction on the length of complexes in
$\Ch^b(\Coh(\sX))$, we conclude the proof.
\end{proof}

\begin{lem}\label{lem:P-duality}
Let $\sX$ be a Noetherian regular stack. Then the canonical map
of spectra $K(\sX) \to G(\sX)$ is a homotopy equivalence.
\end{lem}
\begin{proof}
As for schemes \cite[Theorem~3.21]{TT}, it suffices to show that every
cohomologically bounded pseudo-coherent complex $E^{\bullet}$ on $\sX$ is
perfect. Let $u: U \to \sX$ be a smooth atlas such that $U$ is affine.
Since $(u:U \to \sX)$ is an object of $\LisEt(\sX)$,
the functor $u^*: \Mod(\sX) \to \Mod(U)$ is exact and preserves
coherent sheaves. It follows that $u^*(E^{\bullet})$ is a cohomologically 
bounded pseudo-coherent complex on $U$. Since $U$ is a regular scheme, 
we conclude 
from the proof of \cite[Theorem~3.21]{TT} that $u^*(E^{\bullet})$ is perfect. 
But this implies that $E^{\bullet}$ is perfect on $\sX$.  
\end{proof}

\begin{thm}\label{thm:smooth-K}
Let $\sX$ be a Noetherian regular stack such that $D_{\qc}(\sX)$ is compactly
generated. Then the following hold.
\begin{enumerate}
\item
The canonical maps $K(\sX) \to G(\sX) \leftarrow G^{\rm naive}(\sX)$
are homotopy equivalences.
\item
For any vector bundle $\sE$ on $\sX$ and any $\sE$-torsor $\pi: \sY \to \sX$, 
the pullback map 
$K(\sX) \to K(\sY)$ is a homotopy equivalence.
\item The canonical morphisms of spectra
\[
K(\sX) \to K^B(\sX) \to KH(\sX)
\]
are homotopy equivalences. In particular, $K_i(\sX) = KH_i(\sX) = 0$
for $i < 0$.
\end{enumerate}
\end{thm}

\begin{proof}
Part (1) of the theorem follows directly from Lemmas~\ref{lem:G-thry} and
~\ref{lem:P-duality}.
As shown in \cite[Theorem~2.11]{Merkurjev}, there exists a short
exact sequence of vector bundles
\[
0 \to \sE \to \sW \xrightarrow{\phi} \A^1_{\sX} \to 0
\]
such that $\sY = \phi^{-1}(1)$.
In particular, $\sY$ is the complement of the projective bundle
$\P(\sE)$ in $\P(\sW)$. It follows from our hypothesis and
\lemref{lem:Quasi-proj} that $\P(\sW)$ is a Noetherian regular stack
such that $D_{\qc}(\P(\sW))$ is compactly generated. The same holds
for $\P(\sE)$ as well.
The Quillen localization sequence
\begin{equation}\label{eqn:smooth-K-0}
G^{\rm naive}(\P(\sE)) \to G^{\rm naive}(\P(\sW))\to G^{\rm naive}(\sY)
\end{equation}
and the projective bundle formula \cite[Theorem~3.8]{KR-2} now
prove (2).

By (1) and the Quillen localization
sequence for $G^{\rm naive}(-)$ associated to the inclusions
$\sX[T] \inj \P^1_{\sX}$ and $\sX[T^{\pm 1}] \inj \sX[T]$, 
we see that the Bass fundamental theorem holds for $\sX$
and moreover that the sequence (3) of \thmref{thm:Bass-fund} 
is a short exact sequence. 
This implies that one can define $K^B(\sX)$ as in \S\ref{sec:Neg-K},
and moreover that $K(\sX) \xrightarrow{\simeq}K^B(\sX)$. 
On the other hand, it follows from (2) that
$K^B(\sX) \xrightarrow{\simeq} K^B(\sX \times \Delta^n)$ for every $n \ge 0$,
which implies that $K^B(\sX) \xrightarrow{\simeq} KH(\sX)$. The last assertion of (3) holds
because $K(\sX)$ has no negative homotopy groups.
\end{proof}

\section{Cdh-descent for homotopy $K$-theory}
\label{sec:cdh-descent}

We denote by $\mathbf{Stk}'$
the category of stacks $\sX$ satisfying one of the following conditions:
\begin{itemize}
\item $\sX$ has separated diagonal and linearly reductive finite stabilizers.
\item $\sX$ has affine diagonal and linearly reductive almost multiplicative stabilizers.
\end{itemize}
By Corollaries~\ref{cor:CP-exm} and~\ref{cor:Nis-local-aff-1}, for every $\sX\in\mathbf{Stk}'$
and every closed substack $\sZ\subset\sX$ with quasi-compact open complement, the pair
$(\sX,\sZ)$ is perfect. Moreover, by~\thmref{thm:Nis-local-aff}, $\sX$ admits a
 Nisnevich covering by quotient stacks $[U/G]$ where $U$ is affine over an affine scheme
 $S$ and $G$ is a linearly reductive almost multiplicative group scheme over $S$

Note that if $\sX\in\mathbf{Stk}'$ and $\sY\to \sX$ is a representable morphism with affine diagonal,
then also $\sY\in\mathbf{Stk}'$, since the stabilizers of $\sY$ are subgroups of the stabilizers of $\sX$.

\begin{lem}\label{lem:torsor-inv}
	Let $\sX$ be a stack in $\mathbf{Stk}'$
	and let $f:\sY\to \sX$ be a torsor under a vector bundle. Then
	\[
	f^*: KH(\sX) \to KH(\sY)
	\]
	is a homotopy equivalence.
\end{lem}

\begin{proof}
By~\thmref{thm:Nis-local-aff}, there exists a Nisnevich covering $[U/G]\to\sX$
where $U$ is affine over an affine scheme $S$ and $G$ is a linearly reductive $S$-group scheme.
By~\propref{prop:Nis-cd} and~\corref{cor:Nis-descent-KH},
we are reduced to showing that $KH([U/G])\to KH([U/G]\times_\sX\sY)$ is a homotopy equivalence.
But since $U$ and $S$ are affine and $G$ is linearly reductive, the vector bundle torsor
$[U/G]\times_\sX\sY\to[U/G]$ has a section and hence is a vector bundle. The result now follows from~\thmref{thm:KH-Inv} (6).
\end{proof}

The following theorem is our cdh-descent result for the homotopy $K$-theory of stacks.

\begin{thm}\label{thm:cdh-descent}
Let $\sX$ be a stack in $\mathbf{Stk}'$ and let
\begin{equation*}
\xymatrix{
\sE \ar@{^{(}->}[r] \ar[d] & \sY \ar[d]^{p} \\  
\sZ \ar@{^{(}->}[r]^{e} & \sX,}
\end{equation*}
be a Cartesian square where $p$ is a proper representable morphism,
 $e$ is a closed immersion, and $p$ induces an isomorphism $\sY\setminus \sE\cong \sX\setminus\sZ$.
Then the induced square of spectra
 \begin{equation}\label{eqn:cdh-descent-0}
 	\xymatrix@C1pc{
 	KH(\sX) \ar[r]^-{p^*} \ar[d]_{e^*} & KH(\sY) \ar[d] \\
 	KH(\sZ) \ar[r] & KH(\sE)}
 \end{equation}
is homotopy Cartesian.
\end{thm}

\begin{proof}
We proceed in several steps.

\emph{Step 1.} We prove the result under the assumptions that $p$ is projective, that $p$ and $e$ are of finite presentation, and that
 $\sX=[U/G]$ is quasi-affine over $[S/G]$ for some affine scheme $S$ and some linearly reductive
 isotrivial almost multiplicative group scheme $G$ over $S$ 
 (note that such a stack belongs to $\mathbf{Stk}'$ and has the resolution property).
 Since $G$ is finitely presented, we can write $U$ as an inverse limit of quasi-affine $G$-schemes of finite presentation over $S$.
 By \thmref{thm:KH-Inv} (5), $KH$ transforms such limits into homotopy colimits. 
 Since homotopy colimits of spectra commute with homotopy pullbacks, we can assume that $U$ is finitely presented over $S$.
 We are now in the situation of \cite[Theorem 1.3]{Hoyois-2}, and we deduce that \eqref{eqn:cdh-descent-0} is a homotopy Cartesian square
 for the $KH$-theory defined in \cite{Hoyois-2} (more precisely for the presheaf of spectra $KH_{[S/G]}$ defined in \cite[\S4]{Hoyois-2}). 
 But the latter agrees with the $KH$-theory defined in this paper, by~\lemref{lem:torsor-inv}.
 
  \emph{Step 2.} We prove the result under the assumption that $p$ is projective and that $\sX$ is as in Step 1.
   Since every quasi-coherent sheaf on $\sX$
  is the union of its finitely generated quasi-coherent subsheaves \cite{Rydh-approx}, we can write $\sZ$ as a filtered intersection
  of finitely presented closed substacks of $\sX$. By continuity of $KH$ (\thmref{thm:KH-Inv} (5)), we can therefore assume that $e$ is finitely presented.
  In particular, $\sU=\sX\setminus\sZ$ is quasi-compact.
  Since $\sY$ is projective over $\sX$, it is a closed substack of $\P(\sF)$ for some finitely generated quasi-coherent sheaf $\sF$ on $\sX$.
  Since $\sX$ has the resolution property and affine stabilizers, we can write $\sX=[V/\GL_n]$ for some quasi-affine scheme $V$ \cite[Theorem A]{Gross}.
  On such stacks, it is known that every quasi-coherent sheaf is a filtered colimit of finitely presented quasi-coherent sheaves \cite[Theorem A and Proposition 2.10 (iii)]{Rydh-approx2}.
  In particular, $\sF$ is a quotient of a finitely presented sheaf, so we can assume without loss of generality that $\sF$ is finitely presented.
  We can again write $\sY$ as a filtered intersection of finitely presented closed substacks
  $\sY_i\subset\P(\sF)$.
  By \cite[Theorem C (ii)]{Rydh-approx2}, the projection $\sY_i\times_\sX\sU\to\sU$
  is a closed immersion for sufficiently large $i$. But since it has a section, it must be an isomorphism.
 By continuity of $KH$, we can therefore assume that $p$ is finitely presented, and we are thus reduced to Step 1.

 \emph{Step 3.} We prove the result under the assumption that $p$ is projective.
 By~\thmref{thm:Nis-local-aff} and the fact that groups of multiplicative type are
 isotrivial locally in the Nisnevich topology \cite[Remark 2.9]{Hoyois-1}, there exists a Nisnevich covering $\{\sU_i\to\sX\}$ where each $\sU_i$
 is as in Step 1.
 By~\propref{prop:Nis-cd}, there is a sequence of quasi-compact open substacks $\emptyset=\sX_0\subset\dotsb\subset\sX_n=\sX$
  together with Nisnevich squares
 \[
 \xymatrix{
 \sW_j \ar@{^{(}->}[r] \ar[d] & \sV_j \ar[d] \\  
 \sX_{j-1} \ar@{^{(}->}[r] & \sX_j,}
 \]
 where each $\sV_j$ is a quasi-compact open substack of $\coprod_i\sU_i$.
 In particular, each $\sV_j$ and each $\sW_j$ is as in Step 1. 
 Since $KH$ satisfies Nisnevich descent (\corref{cor:Nis-descent-KH}),
 we deduce from Step 2 by a straightforward induction on $j$ that \eqref{eqn:cdh-descent-0} is a homotopy Cartesian square.
  
 \emph{Step 4.} We prove the result in general. 
 As in Step 2, we can assume that the complement of $e$ is quasi-compact.
 By~\corref{cor:birational-Chow}, there exists a projective morphism $\sY'\to\sY$ 
 which is an isomorphism over the complement of $e$ and
 such that the composite $\sY'\to\sY\to\sX$ is projective.
  Consider the squares
 \begin{equation*}
 	\xymatrix@C1pc{
 	KH(\sX) \ar[r]^-{p^*} \ar[d]_{e^*} & KH(\sY) \ar[d] \ar[r] & KH(\sY') \ar[d] \\
 	KH(\sZ) \ar[r] & KH(\sE) \ar[r] & KH(\sE'),}
 \end{equation*}
 where $\sE'=\sE\times_\sY\sY'$.
The right-hand square
and the total square are both homotopy Cartesian by Step 3. Hence, the left-hand square is also
homotopy Cartesian, as desired.
\end{proof}

\section{The vanishing theorems}\label{sec:Neg-vanishing-R}
Our goal now is to use the cdh-descent for homotopy $K$-theory to
prove the vanishing theorems for negative $K$-theory. 
In order to apply cdh-descent, Kerz and Strunk \cite{KS} used the
idea of killing classes in the negative $K$-theory of schemes using
Gruson--Raynaud flatification \cite{GR}. In \S\ref{sec:killing-lemma}, we prove an analog of 
this result for stacks. This is done essentially like in the case of schemes 
where we replace Gruson--Raynaud flatification with
Rydh's flatification theorem for algebraic stacks (\thmref{thm:Rydh}). The vanishing results will be proven in 
\S\ref{sec:Neg-vanishing} and \S\ref{sec:Weibel-DM}.

\subsection{Killing by flatification}\label{sec:killing-lemma}
We shall need the following two preparatory results about quasi-coherent
sheaves on stacks.

\begin{lem}\label{lem:Resolution}
	Let $f: \sY\to\sX$ be a quasi-affine morphism of stacks.
	If $\sX$ satisfies the resolution property, so does $\sY$.
\end{lem}

\begin{proof}
	This is \cite[Lemma 7.1]{HR-2}.
\end{proof}

\begin{lem}\label{lem:Tor-dim}
Let $f: \sY \to \sX$ be a smooth morphism of Noetherian stacks and let $\sF$ be a coherent
sheaf on $\sY$ which is flat over $\sX$. Then $\sF$ has finite tor-dimension over
$\sY$.
\end{lem}
\begin{proof}
	Since the question is smooth-local on $\sX$ and $\sY$, we can assume that $\sX$ and $\sY$ are Noetherian schemes.
	In this case, the result is \cite[Lemma~6]{KS}.
\end{proof}

\begin{prop}[\bf{The Killing Lemma}]\label{prop:Flatification}
	Let $\sX$ be a reduced Noetherian stack and let $f: \sY \to \sX$ be a smooth morphism of finite type
	such that $\sY$ is perfect and satisfies the resolution property.
	Let $n > 0$ be an 
integer and let $\xi \in K_{-n}(\sY)$. Then there exists a sequence of blow-ups
 $u: \sX' \to \sX$ with nowhere dense centers such that
for the induced map $u_\sY: \sY':= \sX' \times_{\sX} \sY \to \sY$,
one has $u_\sY^*(\xi) = 0$ in $K_{-n}(\sY')$.
\end{prop}

\begin{proof}
We repeat the proof of \cite[Proposition~5]{KS} with minor modifications. 
By the construction of negative $K$-theory of perfect stacks
(see Definition~\ref{defn:Neg-K-def}), there exists a surjection
\begin{equation}\label{eqn:Flatification-0}
{\rm Coker}(K_0(\sY \times \A^n) \to K_0(\sY \times \G^n_m)) \surj
K_{-n}(\sY),
\end{equation}
natural in $\sY$. It will therefore suffice to prove that, for any $\xi 
\in K_0(\sY \times \G^n_m)$, there exists a sequence of blow-ups
$u: \sX' \to \sX$ with nowhere dense centers such that 
$u^*_{\sY \times \G_m^n}(\xi)$ lies in the image of the restriction map
\begin{equation}\label{eqn:Flatification-1}
j^*: K_0(\sY' \times \A^n) \to K_0(\sY' \times \G^n_m).
\end{equation}

Since $\sY$ satisfies the resolution property, it follows from \lemref{lem:Resolution}
that $\sY \times \G^n_m$ also satisfies the resolution property.
In particular, $K_0(\sY \times \G^n_m)$ is generated by classes of vector
bundles on $\sY \times \G^n_m$. Since any finite collection of sequences of blow-ups of $\sX$
can be refined by a single such sequence, we can assume that $\xi$ is
represented by a vector bundle $\sE$ on $\sY \times \G^n_m$. 
We can now extend $\sE$ to a coherent sheaf $\sF$ on $\sY \times \A^n$
(by \cite[Theorem A]{Gross} and \cite[Lemma~1.4]{Thom1}).

Choose a commutative square
\begin{equation*}
\xymatrix{
Y \ar[r]^{g} \ar[d]_p & X \ar[d]^q \\
\sY \ar[r]^f & \sX,}
\end{equation*}
where $X$ and $Y$ are algebraic spaces, $p$ and $q$ are smooth surjective maps, and $g$ is smooth of finite type.
By generic flatness (see \cite[Tag 06QR]{SP}), we can find a dense open
subspace $U \subset X$ such that $(q\times\mathrm{id}_{\A^n})^*(\sF)$ is flat over $U$ under the composite map 
$Y \times \A^n \to Y \to X$. Then $U$ induces a dense open substack $\sU\subset\sX$ such that
$\sF$ is flat over $\sU$.
We now apply \thmref{thm:Rydh} to find a sequence of blow-ups $u: \sX' \to \sX$
whose centers are disjoint from $\sU$ such that the strict transform 
$\tilde{\sF}$ of $\sF$ on $\sY' \times \A^n$ is flat over 
$\sX'$. 

We consider the commutative diagram of Cartesian squares
\begin{equation}\label{eqn:Flatification-2}
\xymatrix@C1pc{
\sY' \times \G^n_m \ar[r]^-{j} \ar[d]_{q} & \sY' \times \A^n \ar[r]^-{e'} 
\ar[d]^{w} &
\sY' \ar[r] \ar[d]^{v} & \sX' \ar[d]^{u} \\
\sY \times \G^n_m \ar[r] & \sY \times \A^n \ar[r]^-{e}  &
\sY \ar[r]  & \sX}
\end{equation}
in which the vertical arrows are blow-ups and the horizontal arrows are 
smooth. 
We next recall that the strict transform $\tilde{\sF}$ is defined by the
cokernel of the map 
$\sH^0_{E \times \A^n}(w^*({\sF})) \inj w^*({\sF})$,
where $E \inj \sY'$ is the exceptional locus of the blow-up
and $\sH^0_{(-)}$ is the sheaf of sections with support.
Since $\sF$ restricts to the vector bundle $\sE$ over $\sY \times \G^n_m$,
which in turn is smooth over $\sX$, it follows that 
$j^*(\tilde{\sF}) = q^*(\sE)$, by \cite[Tag 080F]{SP}.

Lemma~\ref{lem:Tor-dim} says that $\tilde{\sF}$ has finite tor-dimension over $\sY' \times \A^n$.
In particular, it defines a class
$[\tilde{\sF}] \in K_0(\sY' \times \A^n)$. Moreover,
we have $[q^*(\sE)] = [j^*(\tilde{\sF})] = j^*([\tilde{\sF}])$.
This finishes the proof.
\end{proof}

\subsection{Vanishing of negative homotopy $K$-theory}\label{sec:Neg-vanishing}
We now use the techniques of cdh-descent and killing by flatification to 
prove our main results on the vanishing of negative $K$-theory of stacks.

\def\Krdim{\mathrm{Kr\,dim}}
\def\covdim{\mathrm{cov\,dim}}
\def\bldim{\mathrm{bl\,dim}}

\begin{defn}\label{defn:cov-dim}
Let $\sX$ be a Noetherian stack.
\begin{enumerate}
	\item The {\sl Krull dimension} $\Krdim(\sX)\in \mathbb N\cup\{\pm\infty\}$ is the Krull dimension
	of the underlying topological space $\lvert\sX\rvert$ (see \cite[Chapter~5]{LM} for
	the definition of $\lvert\sX\rvert$).
	\item The {\sl blow-up dimension} $\bldim(\sX)\in \mathbb N\cup\{\pm\infty\}$ is the supremum of the integers $n\geq 0$ such that
	there exists a sequence $\sX_n\to\sX_{n-1}\to\dotsb\to \sX_0=\sX$ of non-empty stacks
	where each $\sX_i$ is a nowhere dense closed substack of an iterated blow-up of $\sX_{i-1}$.
	\item The {\sl covering dimension} $\covdim(\sX)\in \mathbb N\cup\{\pm\infty\}$ is the least dimension of a scheme $X$
	admitting a faithfully flat finitely presented morphism $X\to\sX$.
	\end{enumerate}
\end{defn}

\begin{lem}\label{lem:dim}
	Let $\sX$ be a Noetherian stack. Then
	\[
	\Krdim(\sX) \leq \bldim(\sX) \leq \covdim(\sX).
	\]
	If $\sX$ is a quasi-DM stack, all three are equal to $\dim(\sX)$.
\end{lem}

\begin{proof}
	The inequality $\Krdim(\sX) \leq \bldim(\sX)$ follows directly from the definitions, since $\Krdim(\sX)$ is the supremum of a subset of the set of integers described in \defref{defn:cov-dim} (2). 
	For the inequality $\bldim(\sX) \leq \covdim(\sX)$, it suffices to prove the following:
	\begin{enumerate}
		\item[(i)] If $\sY\to\sX$ is a blow-up, then $\covdim(\sY)\leq \covdim(\sX)$.
		\item[(ii)] If $\sZ\subset\sX$ is a nowhere dense closed substack, then $\covdim(\sZ)\leq \covdim(\sX)-1$.
	\end{enumerate}
	Let $f: X\to\sX$ be an fppf cover where $X$ is a scheme.
	Then $X$ is Noetherian and $X\times_{\sX}\sY\to X$ is a blow-up of $X$. It follows that $\dim(X\times_{\sX}\sY)\leq \dim(X)$, whence (i).
	By \cite[Tag 04XL]{SP}, the induced map of topological spaces $\lvert f\rvert: \lvert X\rvert \to \lvert\sX\rvert$
	is continuous and open. Using \cite[Tag 03HR]{SP}, we deduce that $X\times_{\sX}\sZ$ is a nowhere dense closed subscheme of $X$.
	It follows that $\dim(X\times_{\sX}\sZ)\leq \dim(X)-1$, whence (ii).
	
	For the last statement, we will prove more generally that the following hold for
	every faithfully flat representable quasi-finite morphism of Noetherian stacks $f:\sY\to\sX$:
	\begin{enumerate}
		\item[(i)] $\dim(\sY)=\dim(\sX)$.
		\item[(ii)] $\Krdim(\sY)\leq \Krdim(\sX)$.
	\end{enumerate}
	If $\sX$ is quasi-DM, we can take $\sY$ to be a scheme and we deduce that $\covdim(\sX)\leq \dim(\sX)$
	and that $\dim(\sX)\leq\Krdim(\sX)$, as desired.
	To prove (i), by definition of the dimension of a stack \cite[Tag 0AFL]{SP}, we are
	immediately reduced to the case where $\sX$ is an algebraic space.
	In this case, the claim follows from \cite[Tags 04NV and 0AFH]{SP}.
	If $Z_0\subset \dotsb \subset Z_n$ is a strictly increasing sequence of irreducible closed
	subsets of $\lvert\sY\rvert$, then $\overline{f(Z_0)}\subset\dotsb\subset \overline{f(Z_n)}$ is a sequence of irreducible
	closed subsets of $\lvert\sX\rvert$. To check that it is strictly increasing, we may again assume that
	$\sX$ is an algebraic space. If the sequence were not strictly increasing, we would have
	a nontrivial specialization in a fiber of $\lvert f\rvert$, which is a discrete space \cite[Tag 06RW]{SP}. 
	This proves (ii).
\end{proof}

\begin{ex}
	Let $k$ be a field, let $n\geq 1$, and let $\sX$ be the stack quotient of 
	$\A^n_k$ by the standard action of the general linear group $GL_n$.
	Then $\Krdim(\sX)=\bldim(\sX)=1$, $\covdim(\sX)=n$, and $\dim(\sX)=n-n^2$.
	We do not know an example where $\Krdim(\sX)\neq\bldim(\sX)$.
\end{ex}

See \S\ref{sec:cdh-descent} for the definition of the category $\mathbf{Stk}'$ appearing in the next theorem.

\begin{thm}\label{thm:Vanishing-KH}
Let $\sX$ be a stack in $\mathbf{Stk}'$ satisfying the resolution property. 
If $\sX$ is Noetherian of blow-up dimension $d$, then $KH_i(\sX) = 0$ for $i < -d$.
\end{thm}
\begin{proof}
We shall prove the theorem by induction on $d$. Since $KH$ is nil-invariant
(take $\sY=\emptyset$ in~\thmref{thm:cdh-descent}), we can assume that $\sX$ is reduced.
We can write $KH(\sX) = \hocolim_n F_n(\sX)$, where
$$F_n(\sX) = \hocolim_{\Delta^{\rm op}_{\le n}} K^B(\sX \times \Delta^{\bullet}).$$
It suffices to show inductively on $n$ that the canonical map
$\pi_i F_n(\sX) \to KH_i(\sX)$ is zero for all $i < -d$.
This is trivial if $n < 0$, so assume $n\geq 0$.

Let $C_{i,n}(\sX)$ denote the cokernel of $\pi_i F_{n-1}(\sX) \to \pi_i F_n(\sX)$.
Since the cofiber of the map $F_{n-1}(\sX) \to F_n(\sX)$ is canonically 
a direct summand of $\Sigma^n K^B(\sX \times \Delta^n)$ (see for instance \cite[Remark 1.2.4.7]{HA}), 
we may identify $C_{i,n}(\sX)$ with a subgroup of $K_{i-n}(\sX \times \Delta^n)$.
By the induction hypothesis, the map $\pi_i F_n(\sX) \to
KH_i(\sX)$ factors through $C_{i,n}(\sX)$:
\begin{equation}\label{eqn:Vanishing-KH-DM-0}
\xymatrix@C.8pc{
\pi_i F_{n-1}(\sX) \ar[r] \ar[dr]_{0} & \pi_i F_n(\sX) \ar[d] \ar@{->>}[r] &
C_{i,n}(\sX) \ar@{-->}[dl]^{\phi_{i,n}} \ar@{^{(}->}[r] & 
K_{i-n}(\sX \times \Delta^n) \\
& KH_i(\sX). & &}
\end{equation}

Hence, it suffices to show that $\phi_{i,n}: C_{i,n}(\sX) \to KH_i(\sX)$ is zero.
Let $\xi \in C_{i,n}(\sX) \subset K_{i-n}(\sX \times \Delta^n)$.
By~\lemref{lem:Resolution}, $\sX\times\Delta^n$ satisfies the resolution property.
By \propref{prop:Flatification}, there exists a sequence of blow-ups
$u: \sX' \to \sX$ with nowhere dense centers such that $u^*(\xi) = 0$ in
$C_{i,n}(\sX') \subset K_{i-n}(\sX' \times \Delta^n)$ (note that $i-n < 0$).
Let $\sZ \subset \sX$ be a nowhere dense closed substack of $\sX$ 
such that $u$ is an isomorphism over the complement of $\sZ$.
By \thmref{thm:cdh-descent}, 
we have a long exact sequence
\begin{equation*}
\cdots \to KH_{i+1}(u^{-1}(\sZ)) \to KH_i(\sX) \to KH_i(\sX') \oplus KH_i(\sZ)
\to \cdots .
\end{equation*}

Note that both $\sZ$ and $u^{-1}(\sZ)$ have blow-up dimension
strictly less than $d$.
By the induction hypothesis, $KH_{i+1}(u^{-1}(\sZ))$ and $KH_i(\sZ)$ are both
zero, so $u^*: KH_i(\sX) \to KH_i(\sX')$ is injective.
Since $\phi_{i,n}$ is natural in $\sX$, we have $u^* \phi_{i,n}(\xi) =
\phi_{i,n} u^*(\xi) = 0$, and we conclude that $\phi_{i,n}(\xi) = 0$.
This finishes the proof.
\end{proof} 

Our next goal is to remove the resolution property assumption from \thmref{thm:Vanishing-KH}.
We will be able to do so under the additional assumption that $\sX$ has finite inertia.
If $X$ is a Noetherian algebraic space, we will denote by $\mathbf{\acute Et}_X$ the category of algebraic spaces over $X$
that are étale, separated, and of finite type.
The following lemma is a Nisnevich variant of \cite[Proposition~3]{KS}.

\begin{lem}\label{lem:Nis-vanishing}
	Let $X$ be a Noetherian algebraic space, let $\sF$ be a Nisnevich sheaf 
	of abelian groups on $\mathbf{\acute Et}_X$, and let
	$r$ be an integer.
	Suppose that, for every point $y\in Y\in \mathbf{\acute Et}_X$ with $\dim\overline{\{y\}}>r$, 
	$\sF(\sO_{Y,y}^h)=0$. Then $H^i_\mathrm{Nis}(X,\sF)=0$ for all $i>r$.
\end{lem}

\begin{proof}
	Let $s\in \sF(X)$ be a section, and
	let $i: Z\hookrightarrow X$ be a closed immersion such that the support of $s$ is $\lvert Z\rvert$, i.e.,
	$\lvert Z\rvert$ is the closed subset of points $x\in X$ such that $s$ is non-zero in every
	open neighborhood of $x$. We claim that
	\[
	\dim(Z) \leq r.
	\]
	Otherwise, let $y\in Z$ be a generic point such that $\dim\overline{\{y\}}>r$.
	Then $i^*(\sF)(\sO_{Z,y})\cong\sF(\sO_{X,y}^h)=0$, so the section $i^*(s)$ of $i^*(\sF)$ vanishes
	on an open neighborhood $Y$ of $y$ in $Z$. 
	This means that $s$ itself vanishes on an étale neighborhood of $Y$.
	Since it also vanishes on $X\setminus Z$ and $\sF$ is a Nisnevich sheaf, it follows that $s$ 
	vanishes on the open $(X\setminus Z)\cup Y$, which is a contradiction.
	
	Let $S$ be a finite set of local sections of $\sF$, and let $\sF_S\subset \sF$ be the subsheaf generated by $S$.
	Let $i_S: X_S\hookrightarrow X$ be a closed immersion such that $\lvert X_S\rvert$ is the union of the closures of the images of the supports of the sections in $S$, and let $j_S: X\setminus X_S\hookrightarrow X$ be the complementary open immersion.
	Then $j_S^*(\sF_S)=0$ since every $s\in S$ is zero over $X\setminus X_S$. Using the gluing short exact sequence
	\[
	0 \to (j_S)_!j_S^*(\sF_S) \to \sF_S \to (i_S)_*i_S^*(\sF_S)\to 0,
	\]
	we deduce that $\sF_S\cong (i_S)_*i_S^*(\sF_S)$.
	If we now write $\sF$ as a filtered colimit $\sF \cong \colim_{S} \sF_S$, we obtain
\[
H^i_{\Nis}(X,\sF) \cong \colim_{S} H^i_\Nis(X,\sF_S) \cong \colim_S H^i_{\Nis}(X_S,i_S^*(\sF_S)).
\]
The last isomorphism holds because $(i_S)_*$ is an exact functor on Nisnevich sheaves of abelian groups.
By our preliminary result, $\dim(X_S)\leq r$.
Since $X_S$ is a Noetherian algebraic space, its
Nisnevich cohomological dimension is bounded by its Krull dimension \cite[Theorem 3.7.7.1]{SAG}.
We therefore have $H^i_{\Nis}(X_S,i_S^*(\sF_S))=0$ for $i>r$, whence $H^i_{\Nis}(X,\sF)=0$ for $i>r$.
\end{proof}

\begin{lem}\label{lem:Nis-sheaf-spectra}
	Let $X$ be a Noetherian algebraic space of finite Krull dimension, 
	let $\sF$ be a presheaf of spectra on $\mathbf{\acute Et}_X$
	satisfying Nisnevich descent, and let $n$ be an integer.
	Suppose that, for every point $y\in Y\in \mathbf{\acute Et}_X$, 
	$\sF(\sO_{Y,y}^h)$ is $(n+\dim\overline{\{y\}})$-connective. Then the spectrum $\sF(X)$ is $n$-connective.
\end{lem}

\begin{proof}
	We can assume without loss of generality that $n=0$.
	Let $\pi_*\sF$ denote the Nisnevich sheaves of homotopy groups of $\sF$.
	Since $X$ is a Noetherian algebraic space of finite Krull dimension,
	its Nisnevich topos has finite homotopy dimension \cite[Theorem 3.7.7.1]{SAG}, so that the
	descent spectral sequence
	\[
	H^p_{\Nis}(X,\pi_q\sF) \Rightarrow \pi_{q-p}\sF(X)
	\]
	is strongly convergent.
	Applying \lemref{lem:Nis-vanishing} to $\pi_q\sF$, we deduce that
	\[
	H^p_{\Nis}(X,\pi_q\sF) = 0
	\]
	for all $p>q$, and we conclude using the above spectral sequence.
\end{proof}

\begin{thm}\label{thm:DM-stack}
Let $\sX$ be a stack in $\mathbf{Stk}'$ with finite inertia, e.g.,
a separated quasi-DM stack with linearly reductive stabilizers.
Assume that $\sX$ is Noetherian of dimension $d$. Then $KH_i(\sX) = 0$ for
$i < -d$.
\end{thm}

\begin{proof}
Let $X$ be the coarse moduli space of $\sX$. Note that $X$ is a Noetherian algebraic space of dimension $d$.
Let $\sF$ be the presheaf of spectra on $\mathbf{\acute Et}_X$ defined by
\[
\sF(Y) = KH(\sX\times_XY).
\]
By \corref{cor:Nis-descent-KH}, $\sF$ satisfies Nisnevich descent on $\mathbf{\acute Et}_X$.
For $y\in Y\in\mathbf{\acute Et}_X$, let $\sX_y^h=\sX\times_X \Spec(\sO_{Y,y}^h)$.
By continuity of $KH$ (\thmref{thm:KH-Inv} (5)), we have 
\[\sF(\sO_{Y,y}^h)\simeq KH(\sX_y^h).\]
By \thmref{thm:Nis-local-aff}, the stack $\sX_y^h$ has the form $[U/G]$, where $U$ is affine and $G$ is
a finite group scheme over $\Spec(\sO_{Y,y}^h)$.
In particular, $\sX_y^h$ belongs to $\mathbf{Stk}'$ and satisfies the resolution property.
Moreover, the dimension of $\sX_y^h$ is at most $d-\dim\overline{\{y\}}$, and it equals its blow-up dimension by \lemref{lem:dim}.
It follows from \thmref{thm:Vanishing-KH} that $\sF(\sO_{Y,y}^h)$ is $(-d+\dim\overline{\{y\}})$-connective.
By \lemref{lem:Nis-sheaf-spectra}, we deduce that $\sF(X)$ is $(-d)$-connective, i.e., that $KH_i(\sX)=0$ for $i<-d$.
\end{proof}

\subsection{Vanishing of negative $K$-theory with coefficients}\label{sec:Weibel-DM}
Let $\sX$ be a perfect stack and let $n\in\Z$.
Recall from \cite[\S5C]{KR-2} that the algebraic $K$-theory of $\sX$ with 
coefficients is defined by
\begin{align*}\label{eqn:i-invert}
K^B(\sX)[1/n] &:= \hocolim (K^B(\sX) \xrightarrow{n}
K^B(\sX) \xrightarrow{n} \cdots),\\
K^B(\sX, \Z/n) &:= K^B(\sX) \wedge {\mathbb{S}}/{n},
\end{align*}
where ${{\mathbb{S}}}/{n}$ is the mod-$n$ Moore spectrum,
and similarly for $KH$.

\begin{prop}\label{prop:KWeibel}
	Let $\sX$ be a perfect stack.
	\begin{enumerate}
		\item If $n$ is nilpotent on $\sX$, then the canonical map $K^B(\sX)[1/n]\to KH(\sX)[1/n]$
		is a homotopy equivalence.
		\item If $n$ is invertible on $\sX$, then the canonical map $K^B(\sX,\Z/n)\to KH(\sX,\Z/n)$
		is a homotopy equivalence.
	\end{enumerate}
\end{prop}

\begin{proof}
	We have shown in the proof of \propref{prop:perfect} (1) that there is a weak equivalence of dg-categories
	\[
	\mathsf D_\perf(\sX\times\A^1) \simeq \mathsf D_\perf(\sX) \otimes \mathsf D_\perf(\A^1).
	\]
	Given this, the proposition follows immediately from \cite[Theorem 1.2]{Tabuada}.
\end{proof}

\begin{thm}\label{thm:Vanishing-K}
Let $\sX$ be a stack in $\mathbf{Stk}'$ satisfying the resolution property
	or having finite inertia.
Assume that $\sX$ is Noetherian of blow-up dimension $d$. 
Then the following hold.
\begin{enumerate}
\item
If $n$ is nilpotent on $\sX$, then $K_i(\sX)[{1}/{n}] = 0$ for any $i < -d $. 
\item
If $n$ is invertible on $\sX$, then $K_i(\sX, {\Z}/n) = 0$ for any $i < -d$.
\end{enumerate}
\end{thm}
\begin{proof}
This follows from Theorems~\ref{thm:Vanishing-KH} and \ref{thm:DM-stack} and~\propref{prop:KWeibel}.
\end{proof}

\section*{Acknowledgments}

We are very grateful to David Rydh for several fruitful discussions about his recent work, which allowed us to significantly enhance the scope of this paper. 

The bulk of this work was completed during the authors’ stay at the Mittag-Leffler Institute as part of the research program ``Algebro-geometric and homotopical methods'', and we would like to thank the Institute and the organizers, Eric Friedlander, Lars Hesselholt, and Paul Arne Østvær, for this opportunity.


\begin{thebibliography}{99}





\bibitem{AOV} D. Abramovich, M. Olsson, A. Vistoli, {\sl Tame stacks in positive
characteristic\/}, Ann. Inst. Fourier (Grenoble), {\bf 58}, (2008), 1057--1091. \ 



\bibitem{Alper} J. Alper, {\sl Good moduli spaces for Artin stacks},
Ann. Inst. Fourier (Grenoble), {\bf 63}, no. 6, (2013), 2349--2402. \

\bibitem{AHR} J. Alper, J. Hall, D. Rydh, {\sl The {\'e}tale local structure of algebraic
stacks\/}, In preparation, (2017), https://people.kth.se/\textasciitilde dary/papers.html. \



\bibitem{AHW}
A. Asok, M. Hoyois, M. Wendt, {\sl Affine representability results in
  {${\mathbb A}^1$}-homotopy theory {I}: vector bundles}, Duke
  Math. J., {\bf 166}, no. 10, (2017), 1923--1953. \

\bibitem{Bass} H. Bass, {\sl Algebraic $K$-theory\/}, W. A. Benjamin, Inc.,
New York-Amsterdam, (1968). \

\bibitem{BIK} D. Benson, S. Iyengar, H. Krause, {\sl Stratifying modular
representations of finite groups\/}, Ann. Math., {\bf 174},  
no. 3, (2011), 1643--1684. \

\bibitem{BGT} A. Blumberg, D. Gepner, G. Tabuada, 
{\sl A universal characterization of higher algebraic $K$-theory},
Geom. Topol., {\bf 17}, (2013), 733--838. \

\bibitem{BZFN} D. Ben-Zvi, J. Francis, D. Nadler, {\sl Integral transforms and
Drinfeld centers in derived algebraic geometry\/}, J. Amer. Math. Soc., 
{\bf 23}, (2010), 909--966. \

\bibitem{Cisinski} D.-C. Cisinski, {\sl Descente par {\'e}clatements
en $K$-th{\'e}orie invariante par homotopie\/},
Ann. Math., {\bf 177}, no. 2, (2013),  425--448. \

\bibitem{CT} D.-C. Cisinski, G. Tabuada {\sl Non connective $K$-theory
via universal invariants\/}, Compos. Math., {\bf 147}, (2011), 1281--1320. \

\bibitem{CHSW} G. Corti{\~n}as, C. Haesemeyer, M. Schlichting, C. Weibel, 
{\sl Cyclic homology, cdh-cohomology and negative $K$-theory\/},  
Ann. Math., (2)  {\bf 167},  no. 2, (2008), 549--573. \


\bibitem{SGA3}
M. Demazure and A. Grothendieck, {\sl Propri{\'e}t{\'e}s g{\'e}n{\'e}rales des
  sch{\'e}mas en groupes}, Lecture Notes in Mathematics, vol. 151, Springer,
  1970. \


\bibitem{Gross} P. Gross, {\sl Tensor generators on schemes and stacks\/},
Algebr.\ Geom., {\bf 4}, no. 4, (2017), 501--522. \

\bibitem{SGA6} P. Berthelot, A. Grothendieck, and L. Illusie, 
{\sl Théorie des Intersections et Théorème de Riemann--Roch\/}, 
Lecture Notes in Mathematics, vol. 225, Springer,
  1971. \

\bibitem{GR} L. Gruson, M. Raynaud, {\sl Crit{\`e}res de platitude et de
projectivit{\'e}, Techniques de platification d'un module\/}, Invent. Math.,
{\bf 13}, (1971), 1--89. \

\bibitem{Haes} C. Haesemeyer, 
{\sl Descent properties of homotopy $K$-theory\/},
Duke Math. J., {\bf 125}, (2004), 589--620. \

\bibitem{HR-1} J. Hall, D. Rydh, {\sl Algebraic groups and compact generation of
their derived categories of representations \/}, Indiana Univ. 
Math. J., {\bf 64}, (2015),  1903--1923. \


\bibitem{HR-2} J. Hall, D. Rydh, {\sl Perfect complexes on algebraic stacks\/},
Compos. Math., {\bf 153}, no. 11, (2017), 2318-2367. \

\bibitem{HR-3} J. Hall, D. Rydh, {\sl Addendum: Étale dévissage, descent 
and pushouts of stacks\/}, J. Algebra, {\bf 498}, (2018), 398--412. \

\bibitem{HNR} J. Hall, A. Neeman, D. Rydh, {\sl One positive and two negative
results for derived categories of algebraic stacks\/},
J. Inst. Math. Jussieu, to appear, (2018), arXiv:1405.1888. \



\bibitem{Hoyois-1} M. Hoyois, {\sl The six operations in equivariant
motivic homotopy theory\/}, Adv. Math., {\bf 305}, (2017), 197--279. \

\bibitem{Hoyois-2} M. Hoyois, {\sl Cdh descent in equivariant
 homotopy $K$-theory\/}, Preprint, (2017),
arXiv:1604.06410. \

\bibitem{Kelly} S. Kelly, {\sl Vanishing of negative $K$-theory in positive
characteristic\/}, Compos. Math., {\bf 150}, (2015), 1425--1434. \

\bibitem{KS} M. Kerz, F. Strunk, {\sl On the vanishing of negative homotopy
$K$-theory\/}, J. Pure Appl. Algebra, {\bf 221}, no. 7, (2017), 1641-1644. \

\bibitem{KST} M. Kerz, F. Strunk, G. Tamme, {\sl Algebraic $K$-theory and 
descent for blow-ups\/}, Invent. Math., {\bf 211}, no. 2, (2018), 523-577. \

\bibitem{Krause} H. Krause, {\sl Localization theory for triangulated categories}, 
in: Triangulated categories, 161--235, London Math. Soc. Lecture Note Ser. 375,
Cambridge Univ. Press, 2010. \





\bibitem{KO} A. Krishna, P.A. {\O}stv{\ae}r,  {\sl Nisnevich descent for 
K-theory of Deligne--Mumford stacks\/}, 
J. $K$-Theory, {\bf 9}, (02), (2012), 291--331.\



\bibitem{KR-2} A. Krishna, C. Ravi, {\sl Algebraic $K$-theory of quotient 
stacks\/}, Ann. K-theory, {\bf 3}, no. 2, (2018), 207--233. \






\bibitem{LM} G. Laumon, L. Moret-Bailly, {\sl Champs alg{\'e}briques\/},
Ergebnisse der Mathematik und ihrer Grenzgebiete, Springer, {\bf 39}, (2000). \





\bibitem{HTT} J. Lurie, {\sl Higher Topos Theory}, Annals of Mathematical Studies, 
vol. 170, Princeton University Press, 2009. \

\bibitem{HA} J. Lurie, {\sl Higher Algebra\/}, Preprint, 
(2016), Available on the author's web page
http://www.math.harvard.edu/\textasciitilde lurie/. \

\bibitem{SAG} J. Lurie, {\sl Spectral Algebraic Geometry\/}, Preprint, 
(2018), Available on the author's web page
http://www.math.harvard.edu/\textasciitilde lurie/. \

\bibitem{Merkurjev} A. Merkurjev, {\sl Equivariant $K$-theory\/},
Handbook of $K$-theory, Springer-Verlag, {\bf 1, 2},  (2005), 925--954. \ 

\bibitem{MV} F. Morel and V. Voevodsky, {\sl $\mathbb{A}^1$-homotopy theory
 of schemes}, Publ. Math. I.H.{\'E}.S., {\bf 90}, (1999), 45--143. \

\bibitem{Neeman-1} A. Neeman, {\sl The Grothendieck duality theorem via
Bousfield's techniques and Brown representability\/}, 
J. Amer. Math. Soc., {\bf 9}, 01, (1996), 205-236. \ 


\bibitem{Rydh-approx2} D. Rydh, {\sl Noetherian approximation of algebraic spaces and stacks\/},
J. Algebra, {\bf 422}, (2015), 105--147. \

\bibitem{Rydh-approx} D. Rydh, {\sl Approximation of sheaves on algebraic stacks\/},
Int. Math. Res. Not., {\bf 2016}, no. 3, (2016), 717--737. \

\bibitem{Rydh-2} D. Rydh, {\sl Equivariant flatification, \'etalification and compactification\/},
In preparation, (2017), https://people.kth.se/\textasciitilde dary/papers.html. \

\bibitem{Schl} M. Schlichting, {\sl Negative K-theory of derived
categories\/}, Math. Z., {\bf 253}, no. 1, 2006, 97-134.  \

\bibitem{SP} Stacks Project authors, {\sl Stacks Project\/}, 
(2017), http://www.stacks.math.columbia.edu. \

\bibitem{Tabuada} G. Tabuada, {\sl $\mathbb{A}^1$-homotopy invariance of algebraic
K-theory with coefficients and du Val singularities\/}, 
Ann. $K$-theory, {\bf 2}, no. 1, (2017), 1--25. \

\bibitem{Thom} R. Thomason, {\sl Algebraic K-theory of group scheme 
actions\/}, Algebraic Topology and Algebraic K-theory, Ann. Math. Stud., 
Princeton, {\bf 113}, (1987), 539--563.\

\bibitem{Thom1} R. Thomason, {\sl Equivariant resolution, linearization, and
 Hilbert's fourteenth problem over arbitrary base schemes\/}, Adv. Math. 
{\bf 65}, no. 1, (1987),16--34.\



\bibitem{TT} R. Thomason, T. Trobaugh, {\sl Higher algebraic K-theory of 
schemes and of derived categories\/}, The Grothendieck Festschrift, Vol. III, 
Progr. Math. {\bf 88}, Birkh{\"a}user Boston, Boston, MA, 1990, 247--435.\




\bibitem{Voev} V. Voevodsky, {\sl Homotopy theory of simplicial sheaves 
in completely decomposable topologies\/}, J. Pure Appl. Algebra,  {\bf 214}
(2010),  1384-1398. \


\bibitem{Weibel-1} C. Weibel, {\sl Homotopy algebraic $K$-theory\/},
AMS Contemp. Math., {\bf 83}, (1988), 461--488. \


\bibitem{Weibel-3} C. Weibel, {\sl The negative $K$-theory of normal surfaces\/},
 Duke Math. J., {\bf 108}, (2001), 1–35. \

\end{thebibliography}
\end{document}